\documentclass[12pt]{article}
\usepackage{latexsym,amssymb,amsmath,a4wide}
\usepackage[english]{babel}
\usepackage{amsmath}
\usepackage{amsthm}
\usepackage{amsfonts}
\usepackage{amssymb}
\usepackage{array}
\usepackage{tabularx}
\usepackage{graphicx}
\usepackage{tikz}
\usetikzlibrary{intersections, calc, angles}

\usepackage[scaled]{helvet}
\usepackage[metapost,mplabels]{mfpic}

\opengraphsfile{mypics}

\newtheorem{thm}{Theorem}[section]
\newtheorem{cor}[thm]{Corollary}

\newtheorem{prop}[thm]{Proposition}
\theoremstyle{definition}

\theoremstyle{remark}
\newtheorem{rem}[thm]{Remark}

\newtheorem{constr}{Construction}

\author{Olivia Reade Jeans\\ Open University, Milton Keynes, MK7 6AA, U.K.\\ olivia.jeans@open.ac.uk\\ ORCID 0000-0002-7598-9680}

\title{Introducing edge-biregular maps}
\begin{document}

\maketitle

\begin{abstract}
We introduce the concept of alternate-edge-colourings for maps, and study highly symmetric examples of such maps. Edge-biregular maps of type $(k,l)$ occur as smooth normal quotients of a particular index two subgroup of $T_{k,l}$, the full triangle group describing regular plane $(k,l)$-tessellations. The resulting colour-preserving automorphism groups can be generated by four involutions. We explore special cases when the usual four generators are not distinct involutions, with constructions relating these maps to fully regular maps. We classify edge-biregular maps when the supporting surface has non-negative Euler characteristic, and edge-biregular maps on arbitrary surfaces when the colour-preserving automorphism group is isomorphic to a dihedral group.
\end{abstract}


\section{Introduction}

A map is an embedding of a connected graph in a surface such that the image of the graph divides the surface into regions which we call faces, while the interior of each face is homeomorphic to an open disc. This paper addresses maps with alternate-edge-colourings and introduces the algebraic theory underlying the most symmetric examples of such maps.

Section 2 introduces the concept of an alternate-edge-colouring for a map, a condition which is equivalent to the medial map being bipartite. We present the monodromy group for this type of map and relate it to the colour-preserving automorphism group. We then study some properties arising from the algebraic background, focussing on the subclass of these maps which have the largest possible colour-preserving automorphism group, maps which we call edge-biregular.

Section 3 addresses special cases of edge-biregular maps, for example when there are semi-edges or boundary components, including constructions relating these to fully regular maps.

A classification of edge-biregular maps supported by surfaces of non-negative Euler characteristic is shown in Section 4.

Section 5 is devoted to a classification of edge-biregular maps on surfaces of negative Euler characteristic where the colour-preserving automorphism group is isomorphic to a dihedral group.

\section{Preliminaries}

\subsection{Alternate-edge-colourings}

A map has an \emph{alternate-edge-colouring} when it is possible to colour the edge set using two colours such that consecutive edges around any given face will be differently coloured and so also two consecutive edges in the cyclic order of edges around any vertex will be assigned different colours. This property is equivalent to the map having a bipartite medial graph. In our diagrams we will use a bold line to denote an edge of one colour (which we call shaded) and dashed lines to indicate edges of the other colour (unshaded). 

We note that a map having an alternate-edge-colouring is also equivalent to being able to define an orientation on the set of corners of a map such that adjacent corners have the opposite orientation. The orientations on the corners can then be defined to be consistent with sweeping the corners always in the same direction with respect to the colouring, for example from the shaded edge to the unshaded edge. In this last sense it is a similar definition to the pseudo-orientable maps introduced by Wilson in \cite{W1}, but we are assigning the orientations to corners of the map rather than vertices. 

An example of a map with an assigned alternate-edge-colouring is shown in Figure \ref{fig1}.

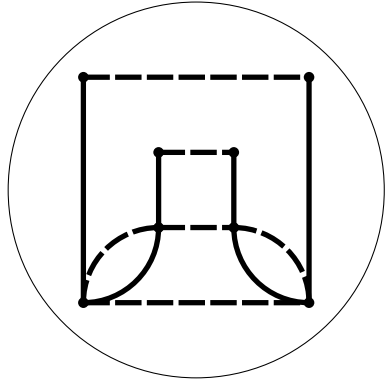
\begin{figure}
\centering
\begin{tikzpicture}
[scale=0.5]
\draw (0,0) circle (5);
\draw [line width=2pt, dash pattern={on 10pt off 2pt}] (-3,-3) -- (3,-3)  (-3,3) -- (3,3) ;
 \draw [line width=2pt] (-3,-3) -- (-3,3)  (3,-3) -- (3,3) ;
 
\draw [line width=2pt, dash pattern={on 10pt off 2pt}] (-1,-1) -- (1,-1)  (-1,1) -- (1,1) ;
 \draw [line width=2pt] (-1,-1) -- (-1,1)  (1,-1) -- (1,1) ; 
 
\draw [line width=2pt, dash pattern={on 10pt off 2pt}] (-1,-1) arc (90:180:2) (3,-3) arc (0:90:2) ;
 \draw [line width=2pt] (-3,-3) arc (270:360:2) (1,-1) arc (180:270:2)   ; 
  \fill (1,1) circle (4pt); 
  \fill (-1,1) circle (4pt); 
    \fill (1,-1) circle (4pt); 
      \fill (-1,-1) circle (4pt); 
   \fill (3,3) circle (4pt); 
      \fill (-3,3) circle (4pt); 
         \fill (3,-3) circle (4pt); 
            \fill (-3,-3) circle (4pt); 
 
 \end{tikzpicture}
\caption{A map on a sphere with an assigned alternate-edge-colouring}
\label{fig1}
\end{figure}

In general, unless stated otherwise, we will assume we are working with maps supported by closed surfaces, that is, maps without boundary components. In this case every vertex of a map with an alternate-edge-colouring must have even valency and all the faces will have even length closed boundary walks. However having even face lengths and valencies is not sufficient for a map to have an alternate-edge-colouring. An example of a map which has only even length faces, and only even degree vertices, is shown in the Figure \ref{fig2}. Any attempt to form an alternate-edge-colouring will result in a contradiction. For example the (single) edge with the double arrow would ``want'' to be coloured with both colours, which is clearly impossible. The same is true for the single arrow edge.

Both examples and non-examples of such maps exist on non-orientable surfaces too. If the left and right edges of the rectangle in Figure \ref{fig2} are identified in the opposite direction from each other then we have a non-orientable map with even valency and even face length for which there is still no alternate-edge-colouring. However, the Klein bottle does support maps with alternate-edge-colourings, as we will see later, in section \ref{classKlein}.

\begin{figure}
\centering
\begin{tikzpicture}[scale=.75]
  \draw [line width=2pt] (0,0) rectangle (12,6);
  \draw [line width=2pt] (2,0)--(2,6);
  \draw [line width=2pt] (10,0)--(10,6);
 \draw  [line width=2pt] (4,0) -- (6,2);
 \draw  [line width=2pt] (8,0) -- (6,2);
 \draw  [line width=2pt] (2,3) -- (6,2);
 \draw  [line width=2pt] (10,3) -- (6,2);
 \draw  [line width=2pt] (2,3) -- (6,4);
 \draw  [line width=2pt] (10,3) -- (6,4);
 \draw  [line width=2pt] (4,6) -- (6,4);
 \draw  [line width=2pt] (8,6) -- (6,4);
  \fill (0,0) circle (4pt);
  \fill (0,3) circle (4pt);
  \fill (0,6) circle (4pt);
  \fill (2,0) circle (4pt);
  \fill (2,3) circle (4pt);
  \fill (2,6) circle (4pt);
  \fill (4,0) circle (4pt);
  \fill (4,6) circle (4pt);
  \fill (6,0) circle (4pt);
  \fill (6,2) circle (4pt);
  \fill (6,4) circle (4pt);
  \fill (6,6) circle (4pt);
  \fill (12,0) circle (4pt);
  \fill (12,3) circle (4pt);
  \fill (12,6) circle (4pt);
  \fill (10,0) circle (4pt);
  \fill (10,3) circle (4pt);
  \fill (10,6) circle (4pt);
  \fill (8,0) circle (4pt);
  \fill (8,6) circle (4pt);
  
  \draw [->, line width=2pt] (0,3.5)--(0,4);
\draw [->, line width=2pt] (12,3.5)--(12,4);

\draw [->>, line width=2pt] (4.5,0)--(5.5,0);
\draw [->>, line width=2pt] (4.5,6)--(5.5,6);
  
 \end{tikzpicture}
\caption{A toroidal map (edges of the rectangle identified in the usual way) with no alternate-edge-colouring}
\label{fig2}
\end{figure}
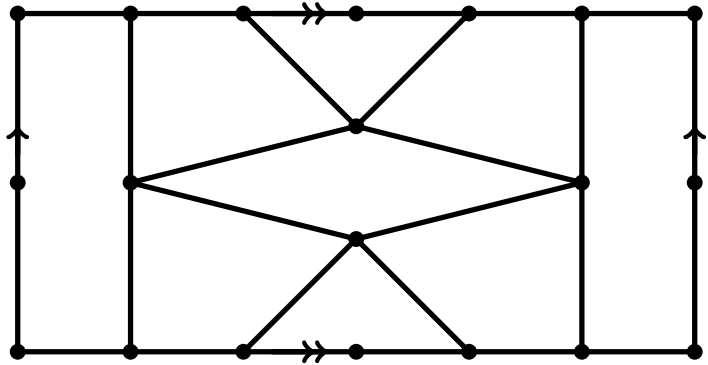

\subsection{The corner-monodromy group}

We define \emph{flags} of a map to be the faces of its barycentric subdivision. We refer to the vertices of the barycentric subdivision as \emph{labels}. Each flag thus has three labels, one for each dimension $0,1$ and $2$, respectively corresponding to the vertex, edge and face of the map with which that flag is associated.

In a map with an assigned alternate-edge-colouring, that is a map which has been given an alternate-edge-colouring, each flag will inherit a natural colouring, shaded or unshaded, according to the colour of the edge with which it associated. The pair of flags which together form a \emph{corner} (that is, the two neighbouring flags from one face which meet in a natural way at a corner of that face) will therefore have one flag of each colour, and the whole map can be decomposed into the set $\mathcal{C}$ of ready-coloured corners, that is corners with an assigned colouring.

The gluing instructions for the elements of the set $\mathcal{C}$ for a map with an assigned alternate-edge-colouring, is in a sense the ready-coloured-corner-monodromy group $G$. The group $G$ acts on the right of $\mathcal{C}$ and is generated by four involutions as follows. Following well-established notation, see \cite{J}, we define the involutions $\mathcal{R}_0, \mathcal{R}_2$ to be the operations which swap every shaded flag with its neighbour respectively along and across a shaded edge. The subscripts in this notation indicate the dimension of the labels of the adjacent shaded flags which are interchanged by the involution. But the action of this group is on the set of two-coloured corners, not flags, so we see $\mathcal{R}_0, \mathcal{R}_2$ swaps every ready-coloured-corner with its unique neighbour respectively along and across a shaded edge, while we define $\mathcal{P}_0, \mathcal{P}_2$ to be the involutions which interchange every ready-coloured-corner with its unique neighbour respectively along and across an unshaded edge.

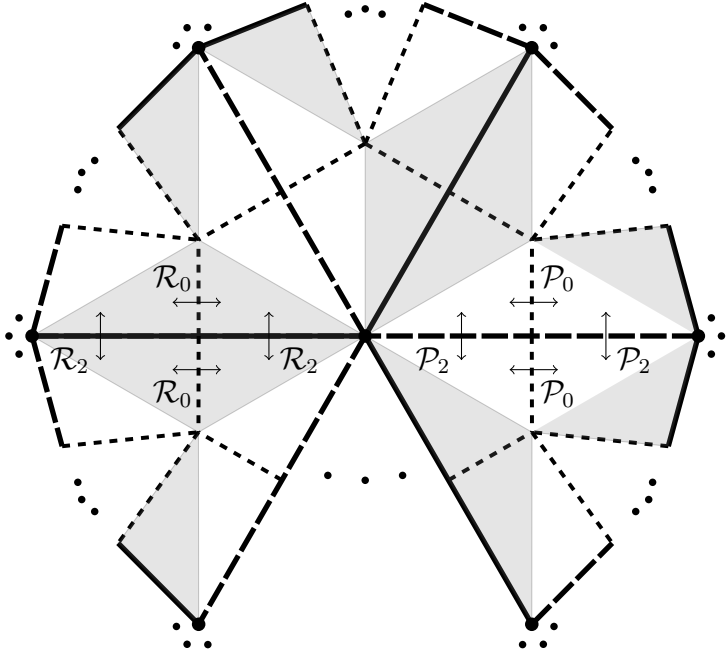
\begin{figure}
\centering
\begin{tikzpicture}
[scale=0.64]
 
  \draw [line width=2pt, dash pattern={on 10pt off 2pt}] (2*3.46,0) -- (0,0) (0,0) -- (120:2*3.46);
 \draw [line width=2pt] (-2*3.46,0) -- (0,0) (0,0) -- (60:2*3.46);
 
\draw [dashed, line width=1.5pt] (90:4) -- (30:4);
\draw [dashed, line width=1.5pt] (90:4) -- (150:4);

\draw [dashed, line width=1.5pt] (0:2*1.73) -- (30:4);

\draw [dashed, line width=1.5pt] (180:2*1.73) -- (150:4);

\draw [dashed, line width=1.5pt] (90:4) -- (80:7);
\draw [dashed, line width=1.5pt] (90:4) -- (100:7);
 \draw [line width=2pt] (120:2*3.46) -- (100:7);
  \draw [line width=2pt, dash pattern={on 10pt off 2pt}] (60:2*3.46) -- (80:7) ;
 \draw [fill=gray, opacity=0.2] (100:7) -- (120:2*3.46) -- (90:4) -- (100:7);

\draw [dashed, line width=1.5pt] (30:4) -- (40:6.7);
\draw [dashed, line width=1.5pt] (150:4) -- (140:6.7);
 \draw  [line width=2pt, dash pattern={on 10pt off 2pt}] (60:2*3.46) -- (40:6.7);
 \draw  [line width=2pt] (120:2*3.46) -- (140:6.7) ;
 \draw[fill=gray, opacity=0.2] (150:4) -- (120:2*3.46) -- (140:6.7) -- (150:4) ;

  \draw [fill=gray, opacity=0.2] (0,0) -- (30:4) -- (60:2*3.46) -- (90:4) -- (0,0);
  
  ***

  \draw [line width=2pt, dash pattern={on 10pt off 2pt}] (-2*3.46,0) -- (0,0) (0,0) -- (-120:2*3.46);
 \draw [line width=2pt] (0,0) -- (-60:2*3.46);
 
\draw [dashed, line width=1.5pt] (-60:3.46) -- (-30:4);
\draw [dashed, line width=1.5pt] (-120:3.46) -- (-150:4);

\draw [dashed, line width=1.5pt] (0:2*1.73) -- (-30:4);
  \draw [fill=gray, opacity=0.2] (0,0) -- (150:4) -- (-180:2*3.46) -- (-150:4) -- (0,0);
\draw [dashed, line width=1.5pt] (-180:2*1.73) -- (-150:4);

\draw [dashed, line width=1.5pt] (-30:4) -- (-40:6.7);
\draw [dashed, line width=1.5pt] (-150:4) -- (-140:6.7);
 \draw  [line width=2pt, dash pattern={on 10pt off 2pt}] (-60:2*3.46) -- (-40:6.7);
 \draw  [line width=2pt] (-120:2*3.46) -- (-140:6.7) ;
 \draw[fill=gray, opacity=0.2] (-150:4) -- (-120:2*3.46) -- (-140:6.7) -- (-150:4) ;

  \draw [fill=gray, opacity=0.2] (0,0) -- (-30:4) -- (-60:2*3.46) -- (0,0);
  
   \draw [line width=2pt, dash pattern={on 10pt off 2pt}] (-2*3.46,0) -- (160:6.7) (-2*3.46,0) -- (-160:6.7);
 \draw [line width=2pt] (2*3.46,0) -- (20:6.7) (2*3.46,0) -- (-20:6.7) ;
 
\draw [dashed, line width=1.5pt] (-20:6.7) -- (330:4);
\draw [dashed, line width=1.5pt] (20:6.7) -- (30:4);

\draw [dashed, line width=1.5pt] (160:6.7) -- (150:4);
\draw [dashed, line width=1.5pt] (200:6.7) -- (210:4);

\draw[fill=gray, opacity=0.2] (0:2*3.46) -- (20:6.7) -- (30:4);
\draw[fill=gray, opacity=0.2] (0:2*3.46) -- (-20:6.7) -- (-30:4);
  
   \fill (0,0) circle (4pt);
   \foreach \x in {0,...,5}
   \fill (60*\x:2*3.46) circle (4pt);

   \draw (270:3) node[circle,fill,inner sep=1pt]{};
  \draw (255:3) node[circle,fill,inner sep=1pt]{};
  \draw (285:3) node[circle,fill,inner sep=1pt]{};
  
     \foreach \x in {-2,...,2}
   \draw (90+60*\x:6.8) node[circle,fill,inner sep=1pt]{}
   (87+60*\x:6.7) node[circle,fill,inner sep=1pt]{}
 (93+60*\x:6.7) node[circle,fill,inner sep=1pt]{};  
  
   \foreach \x in {0,...,5}
   \draw (60*\x:7.4) node[circle,fill,inner sep=1pt]{}
  (3+60*\x:7.2) node[circle,fill,inner sep=1pt]{}
(-3+60*\x:7.2) node[circle,fill,inner sep=1pt]{};

  \draw [<->] (-2, -0.5) node[right] {$\mathcal{R}_2$} -- (-2, 0.5);
  \draw [<->] (2, -0.5) node[left] {$\mathcal{P}_2$} -- (2, 0.5);
  
    \draw [<->] (-4, 0.7) node[above] {$\mathcal{R}_0$} -- (-3, 0.7);
  \draw [<->] (3, 0.7) -- (4, 0.7) node[above] {$\mathcal{P}_0$} ;
 
  \draw [<->] (-5.5, -0.5) node[left] {$\mathcal{R}_2$} -- (-5.5, 0.5);
  \draw [<->] (5, -0.5) node[right] {$\mathcal{P}_2$} -- (5, 0.5);
  
    \draw [<->] (-4, -0.7) node[below] {$\mathcal{R}_0$} -- (-3, -0.7);
  \draw [<->] (3, -0.7) -- (4, -0.7) node[below] {$\mathcal{P}_0$} ;

 \end{tikzpicture}
\caption{The action of $\mathcal{R}_0, \mathcal{R}_2$ and $\mathcal{P}_0,\mathcal{P}_2$ on certain corners of a map with alternate-edge-colouring}
\label{fig3}
\end{figure}

Figure \ref{fig3} shows the action of $\mathcal{R}_0,\mathcal{R}_2,\mathcal{P}_0,\mathcal{P}_2$ on a selection of ready-coloured-corners. Each corner in the diagram is outlined by a bold line, a long-dashed line and two short-dashed lines, and consists of one flag of each colour. The two corners $c$ and $c\mathcal{R}_0$ are thus adjacent along a shaded edge, while $c$ and $c\mathcal{R}_2$ are adjacent across a shaded edge and $c$ is adjacent to $c\mathcal{P}_0$ and $c\mathcal{P}_2$ respectively along and across an unshaded edge.

The group $G = \langle \mathcal{R}_0, \mathcal{R}_2, \mathcal{P}_0, \mathcal{P}_2 \rangle$ for a map with an assigned alternate-edge-colouring which is decomposed into its ready-coloured-corners is then isomorphic to a quotient group of
$\Gamma : = \langle \mathfrak{R}_0, \mathfrak{R}_2 \rangle * \langle \mathfrak{P}_0, \mathfrak{P}_2 \rangle  \cong V_4 * V_4 $
defined by the natural epimorphism $\phi : \Gamma \to G$ where $\phi: \mathfrak{R}_i \to \mathcal{R}_i$ and $\phi: \mathfrak{P}_i \to \mathcal{P}_i$. 
The group $\Gamma$ is the ready-coloured-corner-monodromy group of a universal map for the class of alternate-edge-colourable maps, from which any map with this property can be determined by the natural epimorphism $\phi$.

\subsection{The colour-preserving automorphism group}

We consider the group $H$ of automorphisms acting on the (left of the) set $\mathcal{C}$ which preserve both the structure and the colouring of the map. By this we mean all automorphisms which act such that the images of two neighbouring corners will share the same type of adjacency as their pre-images, that is along or across either a shaded or an unshaded edge. The group $H$ thus consists of all permutations of the set $\mathcal{C}$ which commute with all the gluing instructions in $G$. Hence $H = \{ h \in Sym_\mathcal{C} \; | \; h(c)\mathcal{R}_i = h(c\mathcal{R}_i) \text{ and } h(c)\mathcal{P}_i=h(c\mathcal{P}_i) \text{ for all }c \in \mathcal{C}, i \in \{0,2\}\}$, that is, $H$ is the centraliser of $G$ in the symmetric group acting on the set $\mathcal{C}$.

\subsection{Edge-biregular maps}

We will always assume that all our maps are connected, and hence the structure-preserving condition for our map automorphisms forces the (left) action of $H$ on $\mathcal{C}$ to be semi-regular, that is fixed-point free. The largest possible automorphism group $H$ acting on $\mathcal{C}$, the set of (coloured) corners of a map with assigned alternate-edge-colouring, will be when $H$ acts transitively on the set $\mathcal{C}$. The action being both fixed-point free and transitive means the action is regular and this allows us to identify the elements of the group $H$ with the corners in the set $\mathcal{C}$. We mark an arbitrary corner $C$ which is labelled as the identity element, and refer to the structures incident to this corner as \emph{distinguished}. The automorphisms in $H$ are then generated by the involutions $\langle r_0, r_2, \rho_0, \rho_2\rangle$ which act locally as reflections in the boundaries of this marked corner, (respectively along or across the distinguished shaded edge and along or across the distinguished unshaded edge), while preserving all the adjacency relationships between corners. See Figure \ref{fig4}.

When $H$ acts regularly on the set $\mathcal{C}$ of corners of a map with an assigned alternate-edge-colouring, we say the map is \emph{edge-biregular}. Henceforth, except in some special cases, we will be considering edge-biregular maps with even valency $k$ and even face length $l$ which we describe with the canonical form $(H ; r_0, r_2, \rho_0, \rho_2)$ where the generators act as described above, and $H$ is a group with the partial presentation 
$$H = \langle \; r_0, \; r_2, \;  \rho_0, \; \rho_2 \; \;  |  \; \; r_0^2, \; r_2^2, \;  \rho_0^2, \; \rho_2^2, \; (r_0r_2)^2, \;  (\rho_0\rho_2)^2, \; (r_2 \rho_2)^{k/2}, \;  (r_0\rho_0)^{l/2}, \; \dots \rangle . $$

The exceptional cases, which we address in Section 3, occur when the supporting surface has non-empty boundary components and the resulting edge-biregular maps can then have non-even valency or face length. For example, if in an edge-biregular map a vertex and its incident bold edge, and hence all bold edges, lie on the boundary of the surface then the vertices will necessarily have odd degree, $3$. See Bryant and Singerman \cite{BrSi} for ``Foundations of the theory of maps on surfaces with boundary''.

Each edge-biregular map $M = (H ; r_0, r_2, \rho_0, \rho_2)$ has a \emph{twin} map which is the same as the original map in every respect except the colouring of the edge orbits is switched. Since each generator is associated with one of the colours of the edges, the twin map of $M$ is denoted $W = (H ; \rho_0, \rho_2,  r_0, r_2)$. This is just a matter of relabelling, and does not imply or demand any further relationship between the two structures.

The map automorphisms in $H$ of an edge-biregular map are those which are generated by the four reflections along and across each of two particular adjacent edges of the map, as shown in Figure \ref{fig4}. When considered as acting in the natural way on the flags of the map, the group $H=\langle r_0, r_2, \rho_0, \rho_2\rangle$ partitions the flag set of a given edge-biregular map into two orbits, one containing shaded flags, and the other containing unshaded flags. Thus it is possibly a 2-orbit map, one of those classified by Hubard, Orbani\'{c} and Weiss \cite{HOW}. However an edge-biregular map may not be a 2-orbit map, since there is no structure within our definition which disallows a colour-reversing automorphism of the map. Such an automorphism, if it exists, would fuse the two edge orbits together, and the \textbf{full} automorphism group for this map would then be transitive on flags, making it a fully regular map, and in this case the full automorphism group would then contain $H$ as an index two subgroup. The edge-biregular map $M = (H ; r_0, r_2, \rho_0, \rho_2)$ is thus a fully regular map if and only if there is an involutory automorphism $\psi$ of the group $H$ such that $\psi : r_i \leftrightarrow \rho_i$ for each $i \in \{0,2\}$.

Two edge-biregular maps $M = (H ; r_0, r_2, \rho_0, \rho_2)$ and $M' = (H' ; r'_0, r'_2, \rho'_0, \rho'_2)$ are isomorphic to each other if and only if the mapping $r_i \rightarrow r'_i$ and  $\rho_i \rightarrow \rho'_i$ for $i \in \{0,2\}$ extends to an isomorphism of the group $H$. Thus, an edge-biregular map $M$ is a fully regular map if and only if it is isomorphic to its twin.

Any edge-biregular map which is not fully regular, \textit{is} an example of a $k$-orbit map for $k=2$, see \cite{OPW}. Edge-biregular maps which are not fully regular are the two-orbit maps of type $2_{0,2}$ classified in \cite{HOW}. In \cite{J} Jones also makes special mention of this class of maps as it is the only non-edge-transitive class which arises very naturally in the process of determining the 14 classes of edge-transtitive maps. These types of edge-transitive maps were classified in different terms by Graver and Watkins in \cite{GW} and Wilson in \cite{W2}.

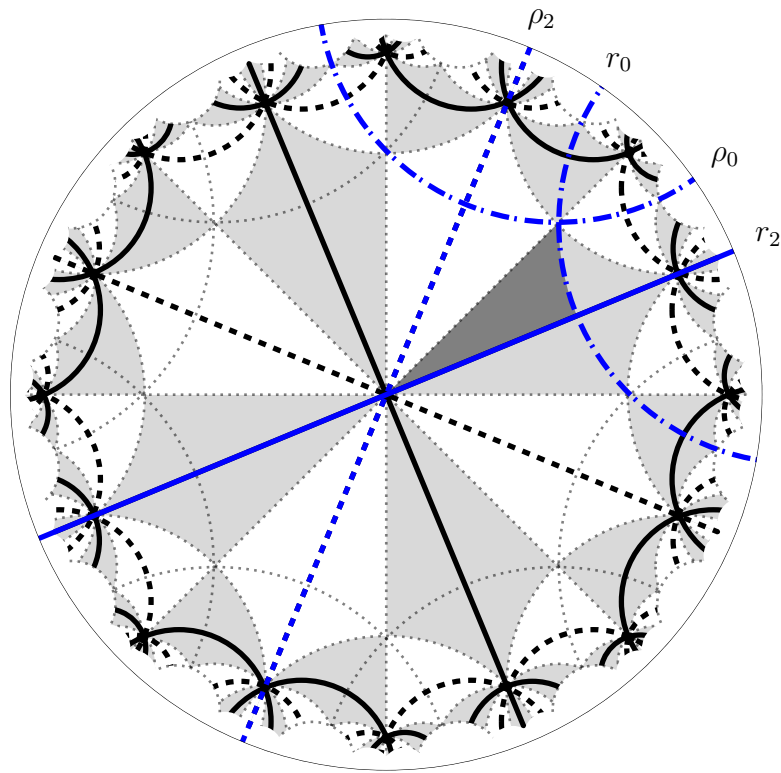
\begin{figure}
\centering
\begin{tikzpicture}[scale=0.5]

\begin{scope}
    \clip (0,0) circle (10);

\begin{scope}
 \clip (0,0)--(0:6.42)--(22.5:8.45)--(45:6.42)--(0,0);
   \fill[fill=gray, opacity=0.3] (0,0) rectangle (45:12) (45:50/6.42+6.42/2) circle (6.42/2-50/6.42) (0:50/6.42+6.42/2) circle (6.42/2-50/6.42);
\end{scope}

\begin{scope}
 \clip (0,0)--(90:6.42)--(112.5:8.45)--(135:6.42)--(0,0);
   \fill[fill=gray, opacity=0.3] (-5,0) rectangle (0,8) (90+45:50/6.42+6.42/2) circle (6.42/2-50/6.42) (90+0:50/6.42+6.42/2) circle (6.42/2-50/6.42);
\end{scope}

\begin{scope}
 \clip (0,0)--(180+0:6.42)--(180+22.5:8.45)--(180+45:6.42)--(0,0);
   \fill[fill=gray, opacity=0.3] (0,0) rectangle (180+45:12) (180+45:50/6.42+6.42/2) circle (6.42/2-50/6.42) (180+0:50/6.42+6.42/2) circle (6.42/2-50/6.42);
\end{scope}

\begin{scope}
 \clip (0,0)--(270+0:6.42)--(270+22.5:8.45)--(270+45:6.42)--(0,0);
   \fill[fill=gray, opacity=0.3] (0,-8) rectangle (5,0) (270+45:50/6.42+6.42/2) circle (6.42/2-50/6.42) (270+0:50/6.42+6.42/2) circle (6.42/2-50/6.42);
\end{scope}
        
 \begin{scope}
  \clip (0,0)--(22.5:5.49)--(45:6.42)--(0,0); 
  \fill[fill=gray] (0,0) rectangle (7,6) (22.5:50/5.49+5.49/2) circle (5.49/2-50/5.49);
\end{scope}

\begin{scope}
 \clip (0,0)--(45:20)--(67.5:20)--(0,0);
   \fill[fill=gray, opacity=0.3]
   (45:50/6.42+6.42/2) circle (6.42/2-50/6.42);
   \fill[fill=white]
     (4*22.5/16+45:50/8.78+8.78/2) circle (8.78/2-50/8.78)
   (13*22.5/16+45:50/8.25+8.25/2) circle (8.25/2-50/8.25);
    \fill[fill=gray, opacity=0.3]
   (90-13*22.5/16:50/8.25+8.25/2) circle (8.25/2-50/8.25)
   (45:50/9.1+9.1/2) circle (9.1/2-50/9.1);
   \fill[fill=white]
   (90:50/6.42+6.42/2) circle (6.42/2-50/6.42)
   (45+22.5/2:50/8.9+8.9/2) circle (8.9/2-50/8.9)
   (-4*22.5/16+45:50/8.78+8.78/2) circle (8.78/2-50/8.78);
\end{scope}

\begin{scope}
 \clip (0,0)--(90+45:20)--(90+67.5:20)--(0,0);
   \fill[fill=gray, opacity=0.3]
   (90+45:50/6.42+6.42/2) circle (6.42/2-50/6.42);
   \fill[fill=white]
     (90+4*22.5/16+45:50/8.78+8.78/2) circle (8.78/2-50/8.78)
   (90+13*22.5/16+45:50/8.25+8.25/2) circle (8.25/2-50/8.25);
    \fill[fill=gray, opacity=0.3]
   (90+90-13*22.5/16:50/8.25+8.25/2) circle (8.25/2-50/8.25)
   (90+45:50/9.1+9.1/2) circle (9.1/2-50/9.1);
   \fill[fill=white]
   (90+90:50/6.42+6.42/2) circle (6.42/2-50/6.42)
   (90+45+22.5/2:50/8.9+8.9/2) circle (8.9/2-50/8.9)
   (90+-4*22.5/16+45:50/8.78+8.78/2) circle (8.78/2-50/8.78);
\end{scope}

\begin{scope}
 \clip (0,0)--(180+45:20)--(180+67.5:20)--(0,0);
   \fill[fill=gray, opacity=0.3]
   (180+45:50/6.42+6.42/2) circle (6.42/2-50/6.42);
   \fill[fill=white]
     (180+4*22.5/16+45:50/8.78+8.78/2) circle (8.78/2-50/8.78)
   (180+13*22.5/16+45:50/8.25+8.25/2) circle (8.25/2-50/8.25);
    \fill[fill=gray, opacity=0.3]
   (180+90-13*22.5/16:50/8.25+8.25/2) circle (8.25/2-50/8.25)
   (180+45:50/9.1+9.1/2) circle (9.1/2-50/9.1);
   \fill[fill=white]
   (180+90:50/6.42+6.42/2) circle (6.42/2-50/6.42)
   (180+45+22.5/2:50/8.9+8.9/2) circle (8.9/2-50/8.9)
   (180+-4*22.5/16+45:50/8.78+8.78/2) circle (8.78/2-50/8.78);
\end{scope}

\begin{scope}
 \clip (0,0)--(270+45:20)--(270+67.5:20)--(0,0);
   \fill[fill=gray, opacity=0.3]
   (270+45:50/6.42+6.42/2) circle (6.42/2-50/6.42);
   \fill[fill=white]
     (270+4*22.5/16+45:50/8.78+8.78/2) circle (8.78/2-50/8.78)
   (270+13*22.5/16+45:50/8.25+8.25/2) circle (8.25/2-50/8.25);
    \fill[fill=gray, opacity=0.3]
   (270+90-13*22.5/16:50/8.25+8.25/2) circle (8.25/2-50/8.25)
   (270+45:50/9.1+9.1/2) circle (9.1/2-50/9.1);
   \fill[fill=white]
   (270+90:50/6.42+6.42/2) circle (6.42/2-50/6.42)
   (270+45+22.5/2:50/8.9+8.9/2) circle (8.9/2-50/8.9)
   (270+-4*22.5/16+45:50/8.78+8.78/2) circle (8.78/2-50/8.78);
\end{scope}

****

\begin{scope}
 \clip (0,0)--(0:20)--(22.5:20)--(0,0);
  \fill[fill=gray, opacity=0.3]
     (4*22.5/16:50/8.78+8.78/2) circle (8.78/2-50/8.78)
   (13*22.5/16:50/8.25+8.25/2) circle (8.25/2-50/8.25);
   \fill[fill=white]
   (45-13*22.5/16:50/8.25+8.25/2) circle (8.25/2-50/8.25)
   (0:50/9.1+9.1/2) circle (9.1/2-50/9.1);
    \fill[fill=gray, opacity=0.3]
   (45:50/6.42+6.42/2) circle (6.42/2-50/6.42)
   (-4*22.5/16:50/8.78+8.78/2) circle (8.78/2-50/8.78);
\end{scope}

\begin{scope}
 \clip (0,0)--(90:20)--(90+22.5:20)--(0,0);
  \fill[fill=gray, opacity=0.3]
     (90+4*22.5/16:50/8.78+8.78/2) circle (8.78/2-50/8.78)
   (90+13*22.5/16:50/8.25+8.25/2) circle (8.25/2-50/8.25);
   \fill[fill=white]
   (90+45-13*22.5/16:50/8.25+8.25/2) circle (8.25/2-50/8.25)
   (90+0:50/9.1+9.1/2) circle (9.1/2-50/9.1);
    \fill[fill=gray, opacity=0.3]
  (90+45:50/6.42+6.42/2) circle (6.42/2-50/6.42)
   (90-4*22.5/16:50/8.78+8.78/2) circle (8.78/2-50/8.78);
\end{scope}

\begin{scope}
 \clip (0,0)--(180:20)--(180+22.5:20)--(0,0);
  \fill[fill=gray, opacity=0.3]
     (180+4*22.5/16:50/8.78+8.78/2) circle (8.78/2-50/8.78)
   (180+13*22.5/16:50/8.25+8.25/2) circle (8.25/2-50/8.25);
   \fill[fill=white]
   (180+45-13*22.5/16:50/8.25+8.25/2) circle (8.25/2-50/8.25)
   (180+0:50/9.1+9.1/2) circle (9.1/2-50/9.1);
   \fill[fill=gray, opacity=0.3]
  (180+45:50/6.42+6.42/2) circle (6.42/2-50/6.42)
   (180-4*22.5/16:50/8.78+8.78/2) circle (8.78/2-50/8.78);
\end{scope}

\begin{scope}
 \clip (0,0)--(270:20)--(270+22.5:20)--(0,0);
  \fill[fill=gray, opacity=0.3]
     (270+4*22.5/16:50/8.78+8.78/2) circle (8.78/2-50/8.78)
   (270+13*22.5/16:50/8.25+8.25/2) circle (8.25/2-50/8.25);
   \fill[fill=white]
   (270+45-13*22.5/16:50/8.25+8.25/2) circle (8.25/2-50/8.25)
   (270+0:50/9.1+9.1/2) circle (9.1/2-50/9.1);
    \fill[fill=gray, opacity=0.3]
   (270+45:50/6.42+6.42/2) circle (6.42/2-50/6.42)
   (270-4*22.5/16:50/8.78+8.78/2) circle (8.78/2-50/8.78);
\end{scope}

****

\begin{scope}
 \clip (0,0)--(0:20)--(-22.5:20)--(0,0);
 \fill[fill=gray, opacity=0.3]
 (0:50/6.42+6.42/2) circle (6.42/2-50/6.42);
  \fill[fill=white]
     (-4*22.5/16:50/8.78+8.78/2) circle (8.78/2-50/8.78)
   (-13*22.5/16:50/8.25+8.25/2) circle (8.25/2-50/8.25);
   \fill[fill=gray, opacity=0.3]
   (-45+13*22.5/16:50/8.25+8.25/2) circle (8.25/2-50/8.25)
   (0:50/9.1+9.1/2) circle (9.1/2-50/9.1);
   \fill[fill=white]
   (-45:50/6.42+6.42/2) circle (6.42/2-50/6.42)
   (4*22.5/16:50/8.78+8.78/2) circle (8.78/2-50/8.78);
\end{scope}

\begin{scope}
 \clip (0,0)--(90:20)--(90-22.5:20)--(0,0);
  \fill[fill=gray, opacity=0.3]
 (90:50/6.42+6.42/2) circle (6.42/2-50/6.42);
  \fill[fill=white]
     (90-4*22.5/16:50/8.78+8.78/2) circle (8.78/2-50/8.78)
   (90-13*22.5/16:50/8.25+8.25/2) circle (8.25/2-50/8.25);
   \fill[fill=gray, opacity=0.3]
   (90-45+13*22.5/16:50/8.25+8.25/2) circle (8.25/2-50/8.25)
   (90+0:50/9.1+9.1/2) circle (9.1/2-50/9.1);
      \fill[fill=white]
   (90-45:50/6.42+6.42/2) circle (6.42/2-50/6.42)
   (90+4*22.5/16:50/8.78+8.78/2) circle (8.78/2-50/8.78);
\end{scope}

\begin{scope}
 \clip (0,0)--(180:20)--(180-22.5:20)--(0,0);
  \fill[fill=gray, opacity=0.3]
 (180:50/6.42+6.42/2) circle (6.42/2-50/6.42);
  \fill[fill=white]
     (180-4*22.5/16:50/8.78+8.78/2) circle (8.78/2-50/8.78)
   (180-13*22.5/16:50/8.25+8.25/2) circle (8.25/2-50/8.25);
   \fill[gray, opacity=0.3]
   (180-45+13*22.5/16:50/8.25+8.25/2) circle (8.25/2-50/8.25)
   (180+0:50/9.1+9.1/2) circle (9.1/2-50/9.1); 
   \fill[fill=white]
   (180-45:50/6.42+6.42/2) circle (6.42/2-50/6.42)
   (180+4*22.5/16:50/8.78+8.78/2) circle (8.78/2-50/8.78);
\end{scope}

\begin{scope}
 \clip (0,0)--(270:20)--(270-22.5:20)--(0,0);
  \fill[fill=gray, opacity=0.3]
 (270:50/6.42+6.42/2) circle (6.42/2-50/6.42);
  \fill[fill=white]
     (270-4*22.5/16:50/8.78+8.78/2) circle (8.78/2-50/8.78)
   (270-13*22.5/16:50/8.25+8.25/2) circle (8.25/2-50/8.25);
   \fill[fill=gray, opacity=0.3]
   (270-45+13*22.5/16:50/8.25+8.25/2) circle (8.25/2-50/8.25)
   (270+0:50/9.1+9.1/2) circle (9.1/2-50/9.1);
      \fill[fill=white]
   (270-45:50/6.42+6.42/2) circle (6.42/2-50/6.42)
   (270+4*22.5/16:50/8.78+8.78/2) circle (8.78/2-50/8.78);
\end{scope}

****

\begin{scope}
 \clip (0,0)--(22.5:20)--(45:20)--(0,0);
 \fill[fill=gray, opacity=0.3]
     (45-4*22.5/16:50/8.78+8.78/2) circle (8.78/2-50/8.78)
   (45-13*22.5/16:50/8.25+8.25/2) circle (8.25/2-50/8.25);
   \fill[fill=white]
   (45-45+13*22.5/16:50/8.25+8.25/2) circle (8.25/2-50/8.25)
   (45:50/9.1+9.1/2) circle (9.1/2-50/9.1);
   \fill[fill=gray, opacity=0.3]
   (45-45:50/6.42+6.42/2) circle (6.42/2-50/6.42)
   (45+4*22.5/16:50/8.78+8.78/2) circle (8.78/2-50/8.78);
\end{scope}

\begin{scope}
 \clip (0,0)--(90+22.5:20)--(90+45:20)--(0,0);
 \fill[fill=gray, opacity=0.3]
     (90+45-4*22.5/16:50/8.78+8.78/2) circle (8.78/2-50/8.78)
   (90+45-13*22.5/16:50/8.25+8.25/2) circle (8.25/2-50/8.25);
   \fill[fill=white]
   (90+45-45+13*22.5/16:50/8.25+8.25/2) circle (8.25/2-50/8.25)
   (90+45:50/9.1+9.1/2) circle (9.1/2-50/9.1);
   \fill[fill=gray, opacity=0.3]
   (90+45-45:50/6.42+6.42/2) circle (6.42/2-50/6.42)
   (90+45+4*22.5/16:50/8.78+8.78/2) circle (8.78/2-50/8.78);
\end{scope}

\begin{scope}
 \clip (0,0)--(180+22.5:20)--(180+45:20)--(0,0);
 \fill[fill=gray, opacity=0.3]
     (180+45-4*22.5/16:50/8.78+8.78/2) circle (8.78/2-50/8.78)
   (180+45-13*22.5/16:50/8.25+8.25/2) circle (8.25/2-50/8.25);
   \fill[fill=white]
   (180+45-45+13*22.5/16:50/8.25+8.25/2) circle (8.25/2-50/8.25)
   (180+45:50/9.1+9.1/2) circle (9.1/2-50/9.1);
   \fill[fill=gray, opacity=0.3]
   (180+45-45:50/6.42+6.42/2) circle (6.42/2-50/6.42)
   (180+45+4*22.5/16:50/8.78+8.78/2) circle (8.78/2-50/8.78);
\end{scope}

\begin{scope}
 \clip (0,0)--(270+22.5:20)--(270+45:20)--(0,0);
 \fill[fill=gray, opacity=0.3]
     (270+45-4*22.5/16:50/8.78+8.78/2) circle (8.78/2-50/8.78)
   (270+45-13*22.5/16:50/8.25+8.25/2) circle (8.25/2-50/8.25);
   \fill[fill=white]
   (270+45-45+13*22.5/16:50/8.25+8.25/2) circle (8.25/2-50/8.25)
   (270+45:50/9.1+9.1/2) circle (9.1/2-50/9.1);
   \fill[fill=gray, opacity=0.3]
   (270+45-45:50/6.42+6.42/2) circle (6.42/2-50/6.42)
   (270+45+4*22.5/16:50/8.78+8.78/2) circle (8.78/2-50/8.78);
\end{scope}

         \draw [line width=2pt] (22.5:10) -- (22.5:-10);
         \draw [dashed, line width=2pt] (67.5:10) -- (67.5:-10);
  \draw [dotted, line width=1pt, opacity=0.5] (90:10) -- (90:-10); 
  \draw [dotted, line width=1pt, opacity=0.5] (45:10) -- (45:-10);
  \draw [dashed, line width=2pt] (90+67.5:10) -- (90+67.5:-10);
  \draw [dotted, line width=1pt, opacity=0.5] (90+45:10) -- (90+45:-10);
  \draw [line width=2pt] (90+22.5:10) -- (90+22.5:-10);
  \draw [dotted, line width=1pt, opacity=0.5] (0:10) -- (0:-10);
  
     \fill (0,0) circle (4pt);

\foreach \x in {0,...,7}
\fill (45*\x+22.5:8.4) circle (4pt);

\foreach \x in {0,...,7}
\fill (45*\x:9.15) circle (4pt);
        
   \foreach \x in {1,...,3}
   \draw[dotted, line width=1pt, opacity=0.5]
   (90*\x+22.5:50/5.49+5.49/2) circle (5.49/2-50/5.49)
    (90*\x+45+22.5:50/5.49+5.49/2) circle (5.49/2-50/5.49);
    
 \foreach \x in {0,...,7}
 \draw[dotted, line width=1pt, opacity=0.5]
  (45*\x:50/9.1+9.1/2) circle (9.1/2-50/9.1);
  
 \foreach \x in {0,...,3}
 \draw[line width=2pt]
  (22.5+90*\x:50/8.4+8.4/2) circle (8.4/2-50/8.4);
   \foreach \x in {0,...,3}
 \draw[dashed, line width=2pt]
  (67.5+90*\x:50/8.4+8.4/2) circle (8.4/2-50/8.4);

 \foreach \x in {0,...,7}
 \draw[dotted, line width=1pt, opacity=0.5]
  (45*\x:50/6.42+6.42/2) circle (6.42/2-50/6.42);

\foreach \x in {0,...,3}
 \draw[dashed, line width=2pt]
 (9*22.5/16+90*\x:50/7.85+7.85/2) circle (7.85/2-50/7.85)
    (45-9*22.5/16+90*\x:50/7.85+7.85/2) circle (7.85/2-50/7.85);
  \foreach \x in {0,...,3}
\draw[line width=2pt]
(45+9*22.5/16+90*\x:50/7.85+7.85/2) circle (7.85/2-50/7.85)
 (-9*22.5/16+90*\x:50/7.85+7.85/2) circle (7.85/2-50/7.85);

\foreach \x in {0,...,3}
 \draw[line width=2pt]
 (2*22.5/16+90*\x:50/9+9/2) circle (9/2-50/9)
 (45-2*22.5/16+90*\x:50/9+9/2) circle (9/2-50/9);
  \foreach \x in {0,...,3}
\draw[dashed, line width=2pt]
 (-2*22.5/16+90*\x:50/9+9/2) circle (9/2-50/9)
 (45+2*22.5/16+90*\x:50/9+9/2) circle (9/2-50/9);

\foreach \x in {0,...,7}
 \draw[dotted, line width=1pt, opacity=0.5]
 (4*22.5/16+45*\x:50/8.78+8.78/2) circle (8.78/2-50/8.78)
 (45-4*22.5/16+45*\x:50/8.78+8.78/2) circle (8.78/2-50/8.78);
 
\foreach \x in {0,...,7}
 \draw[dotted, line width=1pt, opacity=0.5]
 (13*22.5/16+45*\x:50/8.25+8.25/2) circle (8.25/2-50/8.25)
 (45-13*22.5/16+45*\x:50/8.25+8.25/2) circle (8.25/2-50/8.25);

  \foreach \x in {0,...,15}
  \fill[fill=white]
  (22.5*\x+22.5/2:50/8.9+8.9/2) circle (8.9/2-50/8.9);
  \foreach \x in {0,...,15}
 \draw[dotted, line width=1pt, opacity=0.5]
   (22.5*\x+22.5/2:50/8.9+8.9/2) circle (8.9/2-50/8.9);
   
   \foreach \x in {0,...,15}
  \fill[fill=white]
  (22.5+22.5*\x:50/9.6+9.6/2) circle (9.6/2-50/9.6)
(22.5*\x+2*22.5/16:50/9.4+9.4/2) circle (9.4/2-50/9.4)   
   (22.5*\x-2*22.5/16:50/9.4+9.4/2) circle (9.4/2-50/9.4);
 
    *************

\draw[white, line width=5] (0,0) circle (10);
        
         \draw [blue, line width=2pt] (22.5:10) -- (22.5:-10);
         \draw [blue, dashed, line width=2pt] (67.5:10) -- (67.5:-10);
         \draw[blue, dash pattern={on 7pt off 2pt on 1pt off 3pt}, line width=2pt]
  (22.5:50/5.49+5.49/2) circle (5.49/2-50/5.49);
           \draw[blue, dash pattern={on 7pt off 2pt on 1pt off 3pt}, line width=2pt]
  (67.5:50/5.49+5.49/2) circle (5.49/2-50/5.49);
     \draw (0,0) circle (10);

\end{scope}    
 
        \draw (55:10.8) node{$r_0$};
        \draw (22.5:11) node {$r_2$};
        \draw (35:11) node {$\rho_0$};
        \draw (67.5:10.8) node {$\rho_2$};
        
\end{tikzpicture}
\caption{The hyperbolic lines of reflection for the generating automorphisms $r_0, r_2, \rho_0$ and $\rho_2$, shown on part of the infinite edge-biregular map (hyperbolic tessellation) of type $(8,4)$. The marked corner, corresponding to the identity element of $H$, is shaded darker than the others.}

\label{fig4}
\end{figure}

The choice of notation indicates some connection between the groups $G$ and $H$, and indeed there is a very natural relationship between the two. The groups $H$ and $G$ are regular permutation representations on the set of corners of the coloured map acting respectively on the left and right. The embedded Cayley map illustrates this in Figure \ref{fig5}. This is a Cayley graph for the automorphism group $H$, so each of the dark vertices represents an element in $H$, naturally identified with the corners of the edge-biregular map, and the coloured lines are the generating automorphisms $r_0,r_2,\rho_0$ and $\rho_2$. Take the vertex in the corner marked $C$ as the identity element of the group $H$, and consider the element (for example) $h =  \rho_2 r_2 r_0 \rho_0 \in H$. Now, $H$ acts on the left, so $h$ corresponds to the flag which is the image of $C$ after the reflections $\rho_0, r_0, r_2$ and $\rho_2$ are applied in that order. Applying these automorphisms in turn can be a somewhat laborious exercise.

However, one arrives at the same corner as when going from $C$ to $C\mathcal{P}_2\mathcal{R}_2 \mathcal{R}_0\mathcal{P}_0$. This is no coincidence. Notice that the coloured lines, \textit{when looked at in the context of the underlying map shown in gray}, indicate the gluing instructions $ \mathcal{R}_0, \mathcal{R}_2, \mathcal{P}_0, \mathcal{P}_2 $ between the corners. An automorphism $hx \in H$ for some $x \in \{ \rho_2, r_2, r_0, \rho_0 \}$ acts on the marked corner $C$ as follows. Being an automorphism, $x$ commutes with the monodromy group, and $C$ is our marked corner, so $x(C)\mathcal{X}=x(C\mathcal{X})=C$ where $\mathcal{X}$ is the corresponding capital of $x$. This implies that $x(C) = C\mathcal{X}$ and hence $hx(C) = h(C)\mathcal{X}$. By induction we can conclude that, where the corner $C$ is identified with the identity of $H$, the automorphism $h=x_1x_2...x_n \in H$ is identified with the flag $C\mathcal{X}_1\mathcal{X}_2...\mathcal{X}_n$.

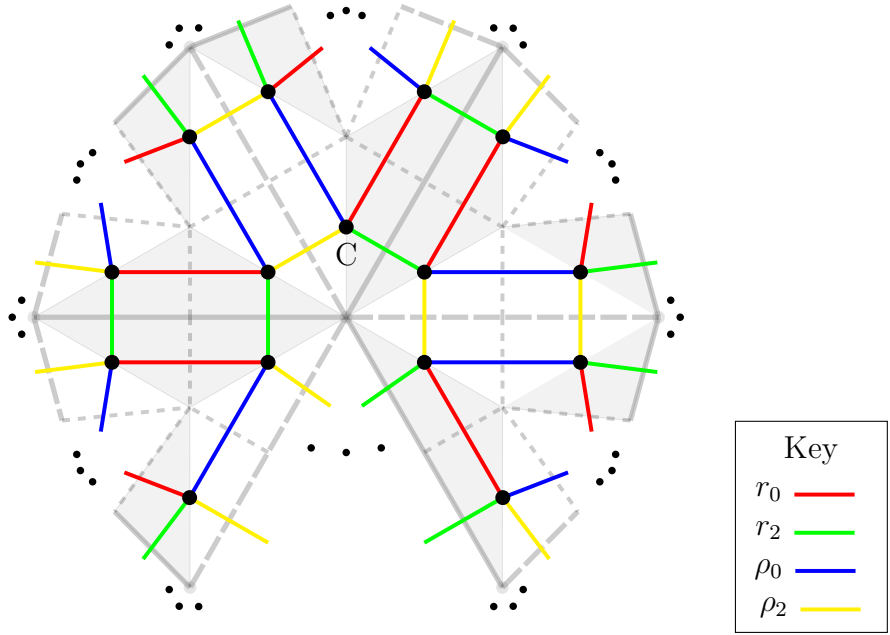
\begin{figure}
\centering
\begin{tikzpicture}
[scale=0.6]
\clip (-12,-7)--(-12,7)--(12,7)--(12,-7) -- (-12,-7);
 
  \draw [line width=2pt, dash pattern={on 10pt off 2pt}, opacity=0.2] (2*3.46,0) -- (0,0) (0,0) -- (120:2*3.46);
 \draw [line width=2pt, opacity=0.2] (-2*3.46,0) -- (0,0) (0,0) -- (60:2*3.46);
 
\draw [dashed, line width=1.5pt, opacity=0.2] (90:4) -- (30:4);
\draw [dashed, line width=1.5pt, opacity=0.2] (90:4) -- (150:4);

\draw [dashed, line width=1.5pt, opacity=0.2] (0:2*1.73) -- (30:4);

\draw [dashed, line width=1.5pt, opacity=0.2] (180:2*1.73) -- (150:4);

\draw [dashed, line width=1.5pt, opacity=0.2] (90:4) -- (80:7);
\draw [dashed, line width=1.5pt, opacity=0.2] (90:4) -- (100:7);
 \draw [line width=2pt, opacity=0.2] (120:2*3.46) -- (100:7);
  \draw [line width=2pt, dash pattern={on 10pt off 2pt}, opacity=0.2] (60:2*3.46) -- (80:7) ;
 \draw [fill=gray, opacity=0.1] (100:7) -- (120:2*3.46) -- (90:4) -- (100:7);

\draw [dashed, line width=1.5pt, opacity=0.2] (30:4) -- (40:6.7);
\draw [dashed, line width=1.5pt, opacity=0.2] (150:4) -- (140:6.7);
 \draw  [line width=2pt, dash pattern={on 10pt off 2pt}, opacity=0.2] (60:2*3.46) -- (40:6.7);
 \draw  [line width=2pt, opacity=0.2] (120:2*3.46) -- (140:6.7) ;
 \draw[fill=gray, opacity=0.1] (150:4) -- (120:2*3.46) -- (140:6.7) -- (150:4) ;

  \draw [fill=gray, opacity=0.1] (0,0) -- (30:4) -- (60:2*3.46) -- (90:4) -- (0,0);
  
  ***

  \draw [line width=2pt, dash pattern={on 10pt off 2pt}, opacity=0.2]  (0,0) -- (-120:2*3.46);
 \draw [line width=2pt, opacity=0.2] (0,0) -- (-60:2*3.46);
 
\draw [dashed, line width=1.5pt, opacity=0.2] (-60:3.46) -- (-30:4);
\draw [dashed, line width=1.5pt, opacity=0.2] (-120:3.46) -- (-150:4);

\draw [dashed, line width=1.5pt, opacity=0.2] (0:2*1.73) -- (-30:4);
  \draw [fill=gray, opacity=0.1] (0,0) -- (150:4) -- (-180:2*3.46) -- (-150:4) -- (0,0);
\draw [dashed, line width=1.5pt, opacity=0.2] (-180:2*1.73) -- (-150:4);

\draw [dashed, line width=1.5pt, opacity=0.2] (-30:4) -- (-40:6.7);
\draw [dashed, line width=1.5pt, opacity=0.2] (-150:4) -- (-140:6.7);
 \draw  [line width=2pt, dash pattern={on 10pt off 2pt}, opacity=0.2] (-60:2*3.46) -- (-40:6.7);
 \draw  [line width=2pt, opacity=0.2] (-120:2*3.46) -- (-140:6.7) ;
 \draw[fill=gray, opacity=0.1] (-150:4) -- (-120:2*3.46) -- (-140:6.7) -- (-150:4) ;

  \draw [fill=gray, opacity=0.1] (0,0) -- (-30:4) -- (-60:2*3.46) -- (0,0);
  
   \draw [line width=2pt, dash pattern={on 10pt off 2pt}, opacity=0.2] (-2*3.46,0) -- (160:6.7) (-2*3.46,0) -- (-160:6.7);
 \draw [line width=2pt, opacity=0.2] (2*3.46,0) -- (20:6.7) (2*3.46,0) -- (-20:6.7) ;
 
\draw [dashed, line width=1.5pt, opacity=0.2] (-20:6.7) -- (330:4);
\draw [dashed, line width=1.5pt, opacity=0.2] (20:6.7) -- (30:4);

\draw [dashed, line width=1.5pt, opacity=0.2] (160:6.7) -- (150:4);
\draw [dashed, line width=1.5pt, opacity=0.2] (200:6.7) -- (210:4);

\draw[fill=gray, opacity=0.1] (0:2*3.46) -- (20:6.7) -- (30:4);
\draw[fill=gray, opacity=0.1] (0:2*3.46) -- (-20:6.7) -- (-30:4);
  
   \fill[gray, opacity=0.2] (0,0) circle (4pt);
   \foreach \x in {0,...,5}
   \fill[gray, opacity=0.2] (60*\x:2*3.46) circle (4pt);

   \draw (270:3) node[circle,fill,inner sep=1pt]{};
  \draw (255:3) node[circle,fill,inner sep=1pt]{};
  \draw (285:3) node[circle,fill,inner sep=1pt]{};
  
     \foreach \x in {-2,...,2}
   \draw (90+60*\x:6.8) node[circle,fill,inner sep=1pt]{}
   (87+60*\x:6.7) node[circle,fill,inner sep=1pt]{}
 (93+60*\x:6.7) node[circle,fill,inner sep=1pt]{};  
  
   \foreach \x in {0,...,5}
   \draw (60*\x:7.4) node[circle,fill,inner sep=1pt]{}
  (3+60*\x:7.2) node[circle,fill,inner sep=1pt]{}
(-3+60*\x:7.2) node[circle,fill,inner sep=1pt]{};

 \draw (0,1.4) node {C};
 
\draw [red, line width=1.5pt] (30:2) -- (3.47,4); 
\draw [red, line width=1.5pt] (-30:2) -- (3.47,-4); 
\draw [red, line width=1.5pt] (0,2) -- (1.73,5); 

\draw [red, line width=1.5pt] (150:2) -- (-3*1.73,1); 
\draw [red, line width=1.5pt] (210:2) -- (-3*1.73,-1); 

\draw [blue, line width=1.5pt] (150:2) -- (-3.47,4); 
\draw [blue, line width=1.5pt] (210:2) -- (-3.47,-4); 
\draw [blue, line width=1.5pt] (90:2) -- (-1.73,5); 

\draw [blue, line width=1.5pt] (30:2) -- (3*1.73,1); 
\draw [blue, line width=1.5pt] (-30:2) -- (3*1.73,-1); 

\draw [green, line width=1.5pt] (1.73,5) -- (3.47,4); 
\draw [green, line width=1.5pt] (-30:2) -- (-80:2); 
\draw [green, line width=1.5pt] (3.47,-4) -- (1.73,-5); 
\draw [green, line width=1.5pt] (90:2) -- (30:2); 

\draw [green, line width=1.5pt] (150:2) -- (210:2); 
\draw [green, line width=1.5pt] (-3*1.73,1) -- (-3*1.73,-1); 

\draw [yellow, line width=1.5pt] (-1.73,5) -- (-3.47,4); 
\draw [yellow, line width=1.5pt] (210:2) -- (-100:2); 
\draw [yellow, line width=1.5pt] (-1.73,-5) -- (-3.47,-4); 
\draw [yellow, line width=1.5pt] (90:2) -- (150:2); 

\draw [yellow, line width=1.5pt] (3*1.73,-1) -- (3*1.73,1); 
\draw [yellow, line width=1.5pt] (-30:2) -- (30:2);

 \draw [yellow, line width=1.5pt] (3.47, 4) -- (50:7);
  \draw [yellow, line width=1.5pt] (3.47, -4) -- (-50:7);
  \draw [yellow, line width=1.5pt] (-3*1.73, 1) -- (170:7);
  \draw [yellow, line width=1.5pt] (-3*1.73, -1) -- (190:7);
  \draw [yellow, line width=1.5pt] (70:7) -- (1.73,5); 
  
   \draw [blue, line width=1.5pt] (3.47, 4) -- (35:6);
  \draw [blue, line width=1.5pt] (3.47, -4) -- (-35:6);
  \draw [blue, line width=1.5pt] (-3*1.73, 1) -- (155:6);
  \draw [blue, line width=1.5pt] (-3*1.73, -1) -- (205:6);
  \draw [blue, line width=1.5pt] (85:6) -- (1.73,5); 
 
 \draw [green, line width=1.5pt] (-3.47, 4) -- (130:7);
  \draw [green, line width=1.5pt] (-3.47, -4) -- (-130:7);
  \draw [green, line width=1.5pt] (3*1.73, 1) -- (10:7);
  \draw [green, line width=1.5pt] (3*1.73, -1) -- (-10:7);
  \draw [green, line width=1.5pt] (110:7) -- (-1.73,5); 
  
   \draw [red, line width=1.5pt] (-3.47, 4) -- (180-35:6);
  \draw [red, line width=1.5pt] (-3.47, -4) -- (180+35:6);
  \draw [red, line width=1.5pt] (3*1.73, 1) -- (25:6);
  \draw [red, line width=1.5pt] (3*1.73, -1) -- (-25:6);
  \draw [red, line width=1.5pt] (95:6) -- (-1.73,5);  
 
 \foreach \x in {0,...,4}
\draw (60*\x -30:2) node[circle,fill,inner sep=2pt]{}; 

\draw (3.47, 4) node[circle,fill,inner sep=2pt]{}; 
\draw (3.47, -4) node[circle,fill,inner sep=2pt]{}; 
\draw (-3.47, 4) node[circle,fill,inner sep=2pt]{}; 
\draw (-3.47, -4) node[circle,fill,inner sep=2pt]{}; 
\draw (-1.73,5) node[circle,fill,inner sep=2pt]{}; 
\draw (1.73,5) node[circle,fill,inner sep=2pt]{}; 
\draw (3*1.73, 1) node[circle,fill,inner sep=2pt]{};  
\draw (3*1.73, -1) node[circle,fill,inner sep=2pt]{}; 
\draw (-3*1.73, 1) node[circle,fill,inner sep=2pt]{}; 
\draw (-3*1.73, -1) node[circle,fill,inner sep=2pt]{}; 
    
    \node[draw, text width=0.1\linewidth,inner sep=2mm,align=center,
      above left] at (current bounding box.south east)
    {Key
    
$r_0$ 
\begin{tikz} \draw [red, line width=1.5pt]  (0,0)--(0.8,0); \end{tikz} 

$r_2$ 
\begin{tikz} \draw [green, line width=1.5pt](0,0)--(0.8,0); \end{tikz}

$\rho_0$ 
\begin{tikz} \draw [blue, line width=1.5pt]  (0,0)--(0.8,0); \end{tikz}     

$\rho_2$ 
\begin{tikz} \draw [yellow, line width=1.5pt]  (0,0)--(0.8,0) ;\end{tikz}
};
 \end{tikzpicture}
\caption{Part of the Cayley map}
\label{fig5}
\end{figure}

\subsection{Algebraic context}

Maps of a given type $(k,l)$ can be obtained as quotients of the universal map of the same type. This universal map consists of the regular tessellation of $l$-gons, $k$ of which meet at each corner, on a simply connected surface. The corresponding $(k,l)$ tessellation is described as hyperbolic, Euclidean or spherical, the name describing the geometry of the underlying simply connected surface.

The universal map of type $(k,l)$ is known \cite{S} to have automorphism group as follows:

$$ T_{k,l}  \; =  \; \langle  \; R_0, R_2, R_1 \; |  \;  R_0^2, R_2^2, R_1^2, (R_0 R_2)^2, (R_1R_2)^k, (R_0R_1)^l  \; \rangle $$

\noindent
The group $T_{k,l}$ is called the full triangle group of type $(k,l)$ and it is finite only for maps of the spherical type, where $1/k+1/l >1/2$. Quotients of the full triangle group by torsion-free normal subgroups of finite index give rise to finite fully regular maps, that is maps where the automorphism group acts transitively on the finite set of flags.

Where $k$ and $l$ are both even, there are seven index two subgroups of $T_{k,l}$, each of which is a source of maps with a given property. These index two subgroups can be identified by listing which of the three generators are included, and we use bar notation to indicate when the corresponding element is not contained in the subgroup. Well-known, and widely studied are the rotary maps which occur as quotients of the index two subgroup in which none of the generators are included, that is $T^+_{k,l} = \langle R_0R_1, R_2R_1 \rangle = \langle \bar{R_0}, \bar{R_1}, \bar{R_2} \rangle$.  Recently Breda, Catalano and \v{S}ir\'{a}\v{n} \cite{BCS} have partially classified the bi-rotary maps which come from quotients of $\langle R_0, \bar{R_2}, \bar{R_1} \rangle$. The maps which are the focus of this paper, that is edge-biregular maps of type $(k,l)$, arise as quotients of the index two subgroup $\langle R_0, R_2, \bar{R_1} \rangle$.

We note that edge-biregular maps can occur as a special case of the 2-regular hypermaps in the ``edge-bipartite'' class as described in Duarte's thesis \cite{D}, where there is a classification for some surfaces of small Euler characteristic. In this paper we allow for maps with semi-edges, making our work more general in this sense.

\section{Special cases of edge-biregular maps}

An edge-biregular map is usually described in terms of the the four generating automorphisms $r_0,r_2,\rho_0, \rho_2$, but there are degenerate cases where either one or more of these is ``missing'' (contributing nothing to the group, and indicating a boundary to the surface) or when the four involutions are not distinct. These possibilities are explored in this section.

\subsection{Edge-biregular maps with semi-edges}
A way in which the four involutions might not be distinct would be if $r_0 = r_2$ and/or $\rho_0 = \rho_2$. This would indicate that the orbit of edges of the corresponding colour consists of semi-edges.

If both orbits of edges are semi-edges then we have a semistar. These maps have only one vertex, and, assuming there are more than two corners in the map, the group $H = \langle r_2, \rho_2 \rangle$ will be dihedral. The supporting surface, if without boundary, must be a sphere, and these even valency semistar spherical maps are both edge-biregular and also fully regular. Examples of semistar edge-biregular maps also exist on a disc, and we explore surfaces with boundary components in the following section.

Henceforth we will assume, up to twinness, that at least one of the orbits of edges does not consist of semi-edges. When there are no semi-edges in an edge-biregular map we say it is a \emph{proper} edge-biregular map.

\subsection{Edge-biregular maps with boundary components}

If a map is on a surface which is not closed, that is when the surface has non-empty boundary components, then at least one of the edges of a corner region will lie on the boundary of the supporting surface. The resulting maps are reminiscent of maps with holes, see page 109 in \cite{CM}, and holey maps, see \cite{AGS}. Figure \ref{fig6} shows a corner region, and ways that corner regions can meet the boundary of a surface. Of course, a corner region could have up to four of its edges along the boundary, see Figures \ref{fig7}, \ref{fig8}. For clarity of the diagrams, we have drawn the boundary lines extending beyond the corner regions. In actuality, due to the map being on a 2-manifold with boundary (where every point on a boundary has a neighbourhood homeomorphic to a half-disc), wherever two boundaries apparently cross, the boundary will in fact be one continuous boundary around this part of the corner region.

\begin{figure}
\centering
\begin{tikzpicture}
  \draw [fill=gray, opacity=0.2] (-2.5,1.8) -- (-2.5,0) -- (-1,1.8) -- (-2.5,1.8);
  \draw [opacity=0.1] (-1,0) rectangle (-2.5,1.8);
  \draw [dashed, line width=2pt] (-1,0) -- (-2.5,0);
  \fill (-2.5,0) circle (3pt);
  \draw [line width=2pt] (-2.5,1.8) -- (-2.5,0);
  
    \draw [fill=gray, opacity=0.2] (0,1.8) -- (0,0) -- (1.5,1.8) -- (0,1.8);
  \draw [opacity=0.1] (1.5,0) rectangle (0,1.8);
  \draw [dashed, line width=2pt] (1.5,0) -- (0,0);
  \fill (0,0) circle (3pt);
  \draw [line width=2pt] (0,1.8) -- (0,0);
  
  \draw [dashed, dash pattern={on 10pt off 2pt}] (-0.5,0) --(10,0);

      \draw [fill=gray, opacity=0.2] (2.8+4,1.8) -- (2.8+2.5,0) -- (2.8+4,0) -- (2.8+4,1.8);
  \draw [opacity=0.1] (2.8+2.5+1.5,0) rectangle (2.8+2.5+0,1.8);
  \draw [dashed, line width=2pt] (2.8+2.5+1.5,1.8) -- (2.8+2.5+0,1.8);
  \fill (2.8+4,1.8) circle (3pt);
  \draw [line width=2pt] (2.8+4,1.8) -- (2.8+4,0);
  
        \draw [fill=gray, opacity=0.2] (2.5,1.5) -- (2.5,0) -- (-2.5+6.8,0) -- (-2.5+5,1.5);
  \draw [opacity=0.1] (-2.5+5+1.8,0) rectangle (-2.5+5+0,1.5);
  \draw [dashed, line width=2pt] (-2.5+6.8,1.5) -- (-2.5+6.8,0);
  \fill (-2.5+6.8,0) circle (3pt);
  \draw [line width=2pt] (-2.5+5+1.8,0) -- (-2.5+5+0,0);
  
          \draw [fill=gray, opacity=0.2] (2.8+6.8,1.5) -- (2.8+5,1.5) -- (2.8+6.8,0) -- (2.8+6.8,1.5);
  \draw [opacity=0.1] (2.8+5+1.8,0) rectangle (2.8+5+0,1.5);
  \draw [dashed, line width=2pt] (2.8+5+0,1.5) -- (2.8+5+0,0);
  \fill (2.8+5+0,1.5) circle (3pt);
  \draw [line width=2pt] (2.8+5+1.8,1.5) -- (2.8+5+0,1.5);
  
 \end{tikzpicture}
\caption{A corner, and corners meeting the surface boundary, types (a), (b), (c), and (d)}
\label{fig6}
\end{figure}
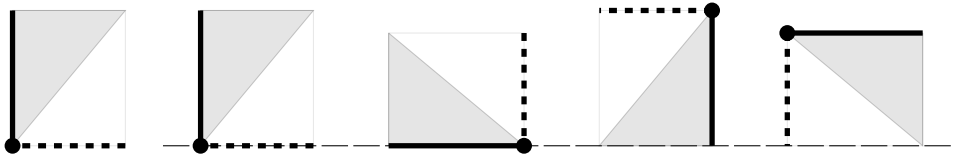

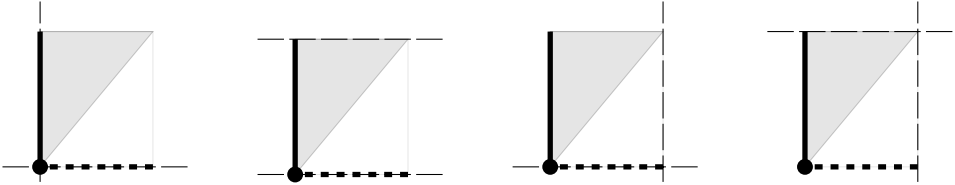
\begin{figure}
\centering

  \begin{tikzpicture}
    \draw [fill=gray, opacity=0.2] (0,1.8) -- (0,0) -- (1.5,1.8) -- (0,1.8);
  \draw [opacity=0.1] (1.5,0) rectangle (0,1.8);
  \draw [dashed, line width=2pt] (1.5,0) -- (0,0);
  \fill (0,0) circle (3pt);
  \draw [line width=2pt] (0,1.8) -- (0,0);
    \draw [dashed, dash pattern={on 10pt off 2pt}] (-0.5,0) --(2,0)  (0,-0.2) -- (0,2.2);
     \end{tikzpicture}
        \hspace{0.6cm}
    \begin{tikzpicture}
    \draw [fill=gray, opacity=0.2] (0,1.8) -- (0,0) -- (1.5,1.8) -- (0,1.8);
  \draw [opacity=0.1] (1.5,0) rectangle (0,1.8);
  \draw [dashed, line width=2pt] (1.5,0) -- (0,0);
  \fill (0,0) circle (3pt);
  \draw [line width=2pt] (0,1.8) -- (0,0);
    \draw [dashed, dash pattern={on 10pt off 2pt}] (-0.5,0) --(2,0)  (-0.5,1.8) --(2,1.8) ;
     \end{tikzpicture}
        \hspace{0.6cm}
          \begin{tikzpicture}
    \draw [fill=gray, opacity=0.2] (0,1.8) -- (0,0) -- (1.5,1.8) -- (0,1.8);
  \draw [opacity=0.1] (1.5,0) rectangle (0,1.8);
  \draw [dashed, line width=2pt] (1.5,0) -- (0,0);
  \fill (0,0) circle (3pt);
  \draw [line width=2pt] (0,1.8) -- (0,0);
    \draw [dashed, dash pattern={on 10pt off 2pt}]   (-0.5,0) --(2,0)  (1.5,-0.2) -- (1.5,2.2) ;
     \end{tikzpicture}
        \hspace{0.6cm}
          \begin{tikzpicture}
    \draw [fill=gray, opacity=0.2] (0,1.8) -- (0,0) -- (1.5,1.8) -- (0,1.8);
  \draw [opacity=0.1] (1.5,0) rectangle (0,1.8);
  \draw [dashed, line width=2pt] (1.5,0) -- (0,0);
  \fill (0,0) circle (3pt);
  \draw [line width=2pt] (0,1.8) -- (0,0);
    \draw [dashed, dash pattern={on 10pt off 2pt}]  (-0.5,1.8) --(2,1.8)  (1.5,-0.2) -- (1.5,2.2);
     \end{tikzpicture}

\caption{Corner regions meeting the surface boundary in two ways (up to colouring)}
\label{fig7}
\end{figure}

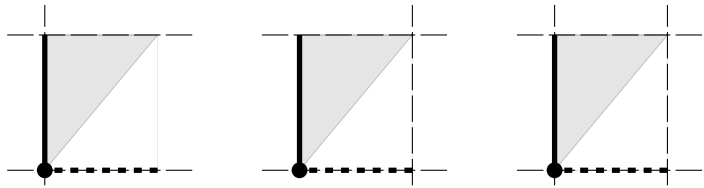
\begin{figure}
\centering
                   \begin{tikzpicture}
    \draw [fill=gray, opacity=0.2] (0,1.8) -- (0,0) -- (1.5,1.8) -- (0,1.8);
  \draw [opacity=0.1] (1.5,0) rectangle (0,1.8);
  \draw [dashed, line width=2pt] (1.5,0) -- (0,0);
  \fill (0,0) circle (3pt);
  \draw [line width=2pt] (0,1.8) -- (0,0);
    \draw [dashed, dash pattern={on 10pt off 2pt}]   (-0.5,0) --(2,0) (-0.5,1.8) --(2,1.8)  (0,-0.2) -- (0,2.2);
     \end{tikzpicture}
      \hspace{0.6cm}
                  \begin{tikzpicture}
    \draw [fill=gray, opacity=0.2] (0,1.8) -- (0,0) -- (1.5,1.8) -- (0,1.8);
  \draw [opacity=0.1] (1.5,0) rectangle (0,1.8);
  \draw [dashed, line width=2pt] (1.5,0) -- (0,0);
  \fill (0,0) circle (3pt);
  \draw [line width=2pt] (0,1.8) -- (0,0);
    \draw [dashed, dash pattern={on 10pt off 2pt}]   (-0.5,0) --(2,0) (-0.5,1.8) --(2,1.8)  (1.5,-0.2) -- (1.5,2.2);
     \end{tikzpicture}       
   \hspace{0.6cm}
                  \begin{tikzpicture}
    \draw [fill=gray, opacity=0.2] (0,1.8) -- (0,0) -- (1.5,1.8) -- (0,1.8);
  \draw [opacity=0.1] (1.5,0) rectangle (0,1.8);
  \draw [dashed, line width=2pt] (1.5,0) -- (0,0);
  \fill (0,0) circle (3pt);
  \draw [line width=2pt] (0,1.8) -- (0,0);
    \draw [dashed, dash pattern={on 10pt off 2pt}]  (0,-0.2) -- (0,2.2) (-0.5,0) --(2,0) (-0.5,1.8) --(2,1.8)  (1.5,-0.2) -- (1.5,2.2);
     \end{tikzpicture}   

\caption{Corners meeting the surface boundary in three (up to colouring) or four ways }
\label{fig8}
\end{figure}

Edge-biregular maps with boundary occur when one or more of the generators (and hence also the relators including them) collapse and hence are missing from the usual group presentation $H = \langle \; r_0, \; r_2, \;  \rho_0, \; \rho_2 \; | \; r_0^2, \; r_2^2, \;  \rho_0^2, \; \rho_2^2 \; (r_0r_2)^2, \;  (\rho_0\rho_2)^2 , \; (r_2 \rho_2)^{k/2}, \;  (r_0\rho_0)^{l/2} , \; \dots \rangle$. For some $x \in \{r, \rho\}$ if $x_0$ is missing then the corresponding coloured edge (and so all such edges) will in fact be a semi-edge to the boundary, while if $x_2$ is missing then all the edges with the corresponding colour will be along the surface boundary. If, for example the unshaded edges are on the surface boundary, then $\rho_2$ is missing and we would denote the edge-biregular map by $M = (H ; r_0, r_2, \rho_0, \bar{\rho_2})$.

Considering an edge-biregular map $M = (H ; r_0, r_2, \rho_0, \bar{\rho_2})$, like type (a) in Figure \ref{fig6}, we have, by regularity of the action of $H$ on corners of the map, all the unshaded edges must be on the boundary of the map. Let us assume, for the moment, that this is the only way in which corner regions meet the surface boundary.  When the bold edges are semi-edges, which means that $r_0 = r_2$, the group $H = \langle r_0, \rho_0 \rangle$ is dihedral, and the edge-biregular map has a single face with the supporting surface being a disc. This is illustrated in Figure \ref{fig9}. When $l=4$ the underlying graph is a ladder so the surface is homeomorphic to either an annulus or a M\"{o}bius strip. These maps can have any number of bold (semi)edges, and generalisations are shown in Figure \ref{fig9}, where the orientability of the second surface is dependent on the identification of edges (and the relators in the group presentation) where the dots are.

\begin{figure}
\centering
\begin{tikzpicture} [scale=0.75]
       \fill[gray, opacity=0.1] (0,0) circle (4);

  \draw [dash pattern={on 8pt off 2pt}] (4,0) arc (0:310:4);
    \draw [dash pattern={on 8pt off 2pt}] (4,0) arc (0:-10:4);
      \draw [dashed, line width=2pt] (4,0) arc (0:310:4);
    \draw [dashed, line width=2pt] (4,0) arc (0:-10:4);
  \foreach \x in {0,...,10}
  \draw [line width=2pt] (30*\x:3) -- (30*\x:4);
  \foreach \x in {0,...,10}
    \fill (30*\x:4) circle (3pt);
    \foreach \x in {1,...,3}
    \fill (-15-10*\x:3.5) circle (1pt);
     \end{tikzpicture}
    \hspace{2cm}
    \begin{tikzpicture} [scale=0.75]
           \fill[gray, opacity=0.1] (0,0) circle (4);
          \fill[white] (0,0) circle (2);
\begin{scope}
\clip (0,0)--(310:5)--(350:5)--(0,0);
   \fill[white]  (0,0) circle (4);
\end{scope}       
  \draw [dash pattern={on 8pt off 2pt}] (4,0) arc (0:310:4);
    \draw [dash pattern={on 8pt off 2pt}] (4,0) arc (0:-10:4); 
    \draw [dash pattern={on 8pt off 2pt}] (2,0) arc (0:310:2);
    \draw [dash pattern={on 8pt off 2pt}] (2,0) arc (0:-10:2); 
    \draw [dashed, line width=2pt] (4,0) arc (0:310:4);
    \draw [dashed, line width=2pt] (4,0) arc (0:-10:4); 
    \draw [dashed, line width=2pt] (2,0) arc (0:310:2);
    \draw [dashed, line width=2pt] (2,0) arc (0:-10:2);
  \foreach \x in {0,...,10}
  \draw  [line width=2pt] (30*\x:2) -- (30*\x:4);
  \foreach \x in {0,...,10}
    \fill (30*\x:4) circle (3pt);
      \foreach \x in {0,...,10}
    \fill (30*\x:2) circle (3pt);
        \foreach \x in {1,...,3}
    \fill (-10-10*\x:3) circle (1pt);
 \end{tikzpicture}
\caption{Edge-biregular maps with unshaded edges along the surface boundary}
\label{fig9}
\end{figure}
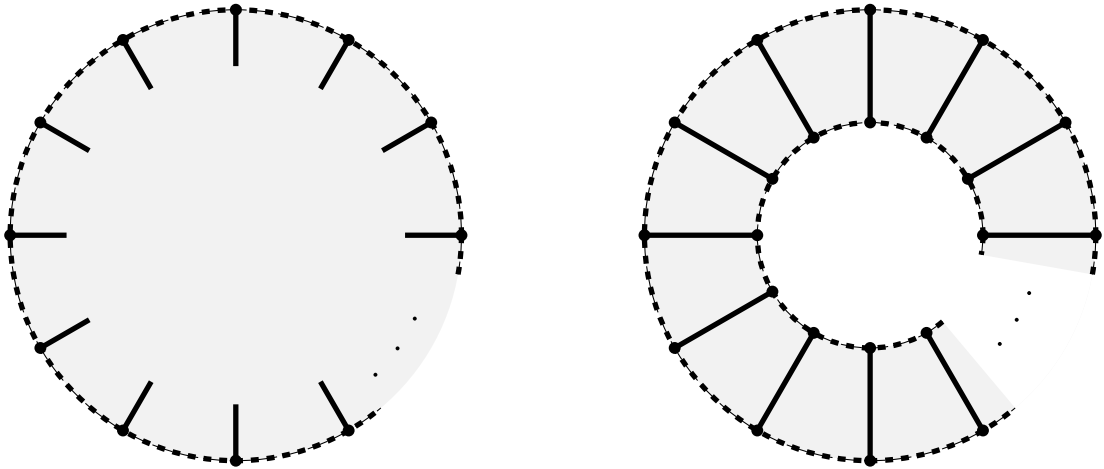

Notice that the vertices of these edge-biregular maps have odd valency. This is no surprise since having the vertices on the surface boundary destroys any hope of a non-trivial rotation around a vertex being an automorphism of the map.

Thinking about these edge-biregular maps in terms of their group presentations yields something even more interesting. Type (a) has the following presentation: $$H = \langle \; r_0, \; r_2, \;  \rho_0 \; \;  | \; \; r_0^2, \; r_2^2, \;  \rho_0^2, \;  (r_0r_2)^2, \;  (r_0\rho_0)^{l/2} , \; \dots \rangle$$ and we are left with $H$ being \emph{a group which is generated by three involutions, two of which commute}. This phrase, in italics, will be very familiar to those who study fully regular maps, as this is precisely how the automorphism group of a fully regular map is often described. We have our group $H$ which looks very similar to (if not exactly the same as!) the standard partial presentation for a regular map with $\rho_0$ in place of $r_1$. This indicates a very close relationship between the two descriptions of maps which is as follows. 

\begin{constr}
Starting with a fully regular map $(G ; r_0, r_2, r_1)$ of type $(k,l)$, see \cite{S} for details, choose an arbitrary vertex and cut out a small disc neighbourhood around that vertex (the disc must be small enough not to include any other vertices nor must it interfere unnecessarily with any edges). Place new vertices where edges meet this new surface boundary, and draw unshaded edges between the new vertices all along this new surface boundary. Repeating this process for all of the vertices of the original fully regular map will yield the well-defined edge-biregular map $(H ; r_0, r_2, \rho_0, \bar{\rho_2})$ of type $(3,2l)$ with $\rho_0 = r_1$. In the other direction, starting with an edge-biregular map with the unshaded edges along a surface boundary, the related regular map can be built by contracting each unshaded edge down into a single vertex. By letting all the unshaded edges (and hence flags, and indeed boundaries) disappear in the process, the reflections which used to act along the unshaded edges (conjugates of $\rho_0$) now simply act as the reflections across the resulting new corners (that is, conjugates of $r_1$) of the implicit fully regular map.
\end{constr}

An edge-biregular map with boundary of type (b) is the twin of a map of type (a), while edge-biregular maps with boundary types (c) and (d) also form twin pairs. Figure \ref{fig10} shows some edge-biregular maps of type (c), the oppositely coloured twins of type (d) edge-biregular maps.

\begin{figure}
\centering
   \begin{tikzpicture} [scale=0.45]
       \fill[gray, opacity=0.1] (0,0) circle (4);
          \fill[white] (0,0) circle (2);
\begin{scope}
\clip (0,0)--(310:5)--(350:5)--(0,0);
   \fill[white]  (0,0) circle (4);
\end{scope}          
          
  \draw [dash pattern={on 8pt off 2pt}] (4,0) arc (0:310:4);
    \draw [dash pattern={on 8pt off 2pt}] (4,0) arc (0:-10:4); 
    \draw [dash pattern={on 8pt off 2pt}] (2,0) arc (0:310:2);
    \draw [dash pattern={on 8pt off 2pt}] (2,0) arc (0:-10:2);
    \draw [dashed, line width=1.5pt] (3,0) arc (0:310:3);
    \draw [dashed, line width=1.5pt] (3,0) arc (0:-10:3);
  \foreach \x in {0,...,10}
  \draw [line width=1.5pt] (30*\x:2) -- (30*\x:4);
  \foreach \x in {0,...,10}
    \fill (30*\x:3) circle (6pt);
        \foreach \x in {1,...,3}
    \fill (-10-10*\x:3) circle (1pt);
 \end{tikzpicture}
 
\vspace{1cm}

    \begin{tikzpicture} [scale=0.45]
    \fill[gray, opacity=0.1] (0,0) circle (6);
  \draw [dash pattern={on 8pt off 2pt}] (0,0) circle (6);
   \foreach \x in {1,...,3}
    \draw [dash pattern={on 8pt off 2pt}] (120*\x:3) circle (1);
  \foreach \x in {1,...,3}
    \draw [dashed, line width=2pt] (0,0) -- (60+120*\x:5);
  \foreach \x in {1,...,3}
    \draw [line width=2pt] (0,0) -- (120*\x:2);
    \draw [line width=2pt] (60:5) -- (3,1);
        \draw [line width=2pt] (60:5) -- (120:3);
            \draw [line width=2pt] (180:5) -- (120:3);
                \draw [line width=2pt] (180:5) -- (240:3);
                    \draw [line width=2pt] (-60:5) -- (240:3);
                        \draw [line width=2pt] (-60:5) -- (3,-1);
                           \foreach \x in {1,...,3}
       \foreach \x in {1,...,3}
    \draw [line width=2pt] (60+120*\x:5) -- (60+120*\x:6);

\foreach \x in {1,...,3} 
    \fill[white] (120*\x:3) circle (1);
 \foreach \x in {1,...,3}
        \draw [dash pattern={on 8pt off 2pt}, fill=white] (120*\x:3) circle (1);
  \draw [dashed, line width=2pt] (0,0) circle (5);
        \foreach \x in {1,...,3}
    \fill (60+120*\x:5) circle (6pt);
    \fill (0,0) circle (6pt);
 \end{tikzpicture}
    \hspace{1.5cm}
 \begin{tikzpicture} [scale=0.45]    
 
\fill[gray, opacity=0.1] (-2,-2) rectangle(10,10);
 
  \foreach \x in {0,...,3}
    \draw [dashed, line width=2pt] (-2+4*\x,-2) -- (-2+4*\x,10);
      \foreach \y in {0,...,3}
    \draw [dashed, line width=2pt] (-2,-2+4*\y) -- (10,-2+4*\y);

\draw [->, line width=2pt] (-2,3.5)--(-2,4);
\draw [->, line width=2pt] (10,3.5)--(10,4);

\draw [->>, line width=2pt] (3.5,-2)--(4,-2);
\draw [->>, line width=2pt] (3.5,10)--(4,10);

    \draw [line width=2pt] (-2,-2) -- (10,10);
    \draw [line width=2pt] (2,-2) -- (10,6);
   \draw [line width=2pt] (6,-2) -- (10,2); 
      \draw [line width=2pt] (-2,2) -- (6,10);
         \draw [line width=2pt] (-2,6) -- (2,10);
         
             \draw [line width=2pt] (-2,2) -- (2,-2);
    \draw [line width=2pt] (2,10) -- (10,2);
   \draw [line width=2pt] (6,10) -- (10,6); 
      \draw [line width=2pt] (-2,6) -- (6,-2);
         \draw [line width=2pt] (-2,10) -- (10,-2);
    
        \foreach \x in {0,...,3}
    \foreach \y in {0,...,3}
    \fill (-2+4*\x,-2+4*\y) circle (6pt);
    \foreach \x in {0,...,2}
    \foreach \y in {0,...,2}
    \fill[white] (4*\x,4*\y) circle (1);
 \foreach \x in {0,...,2}
  \foreach \y in {0,...,2}
        \draw [dash pattern={on 8pt off 2pt}] (4*\x,4*\y) circle (1);

 \end{tikzpicture}
 
\caption{Edge-biregular maps of type (c) with semi-edges meeting the surface boundaries}
\label{fig10}
\end{figure}
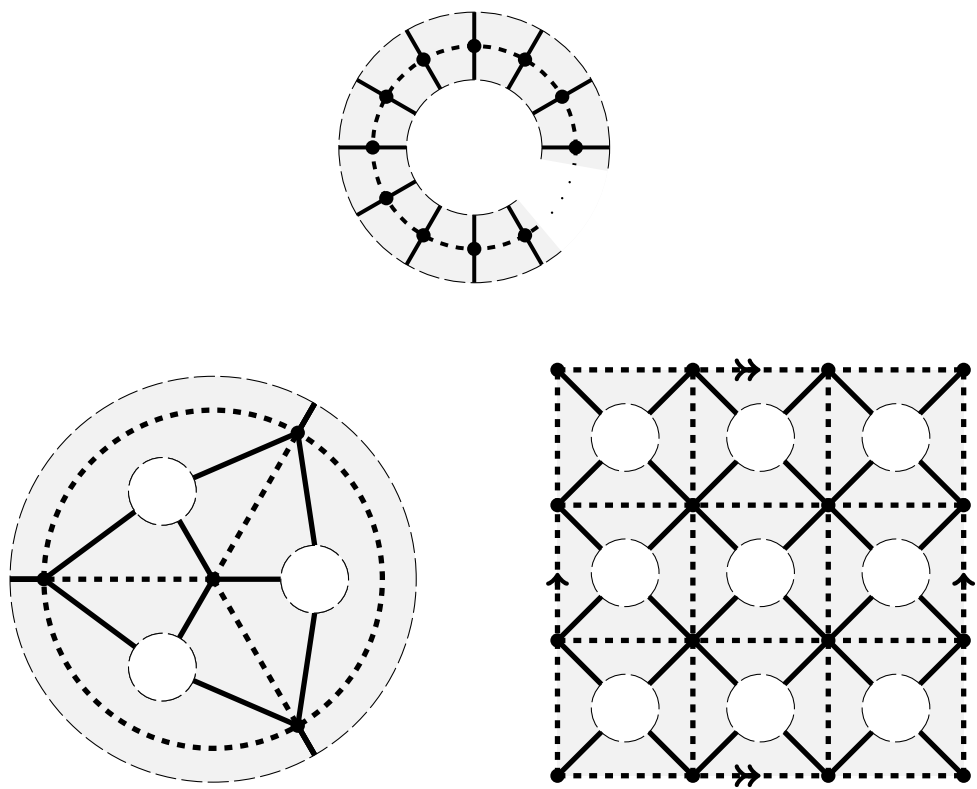

If one of the edges of a particular colour, let us say unshaded, meets the boundary, like (d) in Figure \ref{fig6}, then this edge, and hence all edges of this type are semi-edges to a boundary of the surface. This is an edge-biregular map $(H; r_0, r_2, \bar{\rho_0}, \rho_2, )$. Assuming that a corner region meets the boundary only this way, and the shaded edges are not semi-edges (otherwise we would end up with a semistar), then each face region of the map must have as its boundary: one shaded edge; two unshaded semi-edges; and a section of the surface boundary. This type (d) edge-biregular map thus has presentation: $$H = \langle \; r_0, \; r_2, \;  \rho_2 \; \;  | \; \; r_0^2, \; r_2^2, \; \rho_2^2, \; (r_0r_2)^2, \;  (r_2\rho_2)^{k/2} , \; \dots \rangle$$ and we are again left with $H$ being \emph{a group which is generated by three involutions, two of which commute}. This time we have $\rho_2$ instead of $r_1$, and this yields the following construction. 

\begin{constr}
Starting from a regular map, $(G;r_0,r_2,r_1)$, drawn with all its edges shaded, we can obtain the well-defined implied edge-biregular map $(H; r_0, r_2, \bar{\rho_0}, \rho_2, )$ by cutting a disc out from the interior of each face and drawing an unshaded semi-edge from this new surface boundary to each of the surrounding vertices, thereby letting $r_1=\rho_2$. Alternatively, given an edge-biregular map with unshaded semi-edges to the surface boundary, we can obtain a regular map by deleting all the unshaded semi-edges and contracting each of the boundaries in the surface to ensure the interior of each resulting face is homeomorphic to an open disc. This valid since the boundary component which a semi-edge meets must be the same as the boundary component which the semi-edge from the adjacent corner within the same face meets.
\end{constr}

If a corner meets the boundary in more than one of the ways listed (as in the diagrams in Figures \ref{fig7} and \ref{fig8}) then options are severely restricted: there are at most two involutory generators and so the group $H$ is dihedral or cyclic. The interested reader may wish to verify that these include the single faced maps on a disc $H = \langle r_0, \rho_0 \rangle$, the edge-biregular maps on an orientable band $H =\langle r_2, \rho_0 \rangle$, the 4-cornered map on a disc $H = \langle r_0, r_2 \rangle \cong V_4$ and the semistar on a disc $H = \langle r_2, \rho_2 \rangle$ as well as their twins, and the more trivial maps on a disc which have only one or two corners.

\subsection{Edge-biregular maps with non-distinct generators}

There is another natural, if somewhat trivial way of building an edge-biregular map from any given fully regular map, this time without introducing any boundaries, and that is as follows. 

\begin{constr}\label{constr3}
Given a fully regular map $(G ; r_0, r_2, r_1)$ of type $(k,l)$ drawn with all its edges shaded. At every vertex draw $k$ unshaded semi-edges, so that there is one semi-edge in every corner of the original map. This is thus the edge-biregular map $(H ; r_0, r_2, \rho, \rho) $ where $\rho=\rho_0=\rho_2=r_1$.
\end{constr}

Figure \ref{fig11} shows a cube (a fully regular map) drawn with all shaded edges, along with the edge-biregular maps as described in the three constructions above. If you draw the embedded Cayley graphs for each of these maps, the close relationships between this collection of four maps becomes very clear. We could of course include the corresponding twin maps too, as well as the dual maps, and hence each fully regular map has many edge-biregular children.

\begin{figure}
\centering
\begin{tikzpicture}
[scale=0.42]
\fill[gray, opacity=0.1] (0,0) circle (9);
\draw (0,0) circle (9);
\draw [line width=2pt] (-5,-5) -- (5,-5)  (-5,5) -- (5,5) ;
 \draw [line width=2pt] (-5,-5) -- (-5,5)  (5,-5) -- (5,5) ;
\draw [line width=2pt] (-2,-2) -- (2,-2)  (-2,2) -- (2,2) ;
 \draw [line width=2pt] (-2,-2) -- (-2,2)  (2,-2) -- (2,2) ; 
 
 
\draw [line width=2pt] (-2,-2)--(-5,-5) (2,2)--(5,5) (-2,2)--(-5,5)  (2,-2)--(5,-5);
  \fill (2,2) circle (6pt); 
  \fill (-2,2) circle (6pt); 
    \fill (2,-2) circle (6pt); 
      \fill (-2,-2) circle (6pt); 
   \fill (5,5) circle (6pt); 
      \fill (-5,5) circle (6pt); 
         \fill (5,-5) circle (6pt); 
            \fill (-5,-5) circle (6pt); 
 
 \end{tikzpicture}
 \hspace{0.4cm} 
\begin{tikzpicture}
[scale=0.42]
\fill[gray, opacity=0.1] (0,0) circle (9);
\draw (0,0) circle (9);
\draw [line width=2pt] (-5,-5) -- (5,-5)  (-5,5) -- (5,5) ;
 \draw [line width=2pt] (-5,-5) -- (-5,5)  (5,-5) -- (5,5) ;
\draw [line width=2pt] (-2,-2) -- (2,-2)  (-2,2) -- (2,2) ;
 \draw [line width=2pt] (-2,-2) -- (-2,2)  (2,-2) -- (2,2) ; 
 
 
\draw [line width=2pt] (-2,-2)--(-5,-5) (2,2)--(5,5) (-2,2)--(-5,5)  (2,-2)--(5,-5);
  \fill (2,2) circle (6pt); 
  \fill (-2,2) circle (6pt); 
    \fill (2,-2) circle (6pt); 
      \fill (-2,-2) circle (6pt); 
   \fill (5,5) circle (6pt); 
      \fill (-5,5) circle (6pt); 
         \fill (5,-5) circle (6pt); 
            \fill (-5,-5) circle (6pt); 
 
\draw [line width=2pt, dashed] (2,2)--(1,1) (1,3)--(3,1);  
\draw [line width=2pt, dashed] (-2,2)--(-1,1) (-3,1)--(-1,3); 
\draw [line width=2pt, dashed] (2,-2)--(1,-1) (1,-3)--(3,-1) ; 
\draw [line width=2pt, dashed] (-2,-2)--(-1,-1) (-1,-3)--(-3,-1) ; 
\draw [line width=2pt, dashed] (5,5)--(6,6) (4,3)--(5,5) (5,5)--(3,4) ;  
\draw [line width=2pt, dashed] (-5,5)--(-6,6) (-5,5)--(-4,3) (-5,5)--(-3,4) ; 
\draw [line width=2pt, dashed] (5,-5)--(6,-6) (5,-5)--(4,-3) (5,-5)--(3,-4) ; 
\draw [line width=2pt, dashed] (-5,-5)--(-6,-6) (-5,-5)--(-4,-3) (-5,-5)--(-3,-4) ;

 \end{tikzpicture}
\vspace{0.4cm} 
 
 \begin{tikzpicture}
[scale=0.42]

\fill[gray, opacity=0.1] (0,0) circle (9);
\draw (0,0) circle (9);
\draw [line width=2pt] (-5,-5) -- (5,-5)  (-5,5) -- (5,5) ;
 \draw [line width=2pt] (-5,-5) -- (-5,5)  (5,-5) -- (5,5) ;
\draw [line width=2pt] (-2,-2) -- (2,-2)  (-2,2) -- (2,2) ;
 \draw [line width=2pt] (-2,-2) -- (-2,2)  (2,-2) -- (2,2) ; 
 
 
\draw [line width=2pt] (-2,-2)--(-5,-5) (2,2)--(5,5) (-2,2)--(-5,5)  (2,-2)--(5,-5);

\draw [line width=1pt, dash pattern={on 8pt off 2pt}, fill=white] (2,2) circle (1);  
\draw [line width=1pt, dash pattern={on 8pt off 2pt}, fill=white] (-2,2) circle (1); 
\draw [line width=1pt, dash pattern={on 8pt off 2pt}, fill=white] (2,-2) circle (1); 
\draw [line width=1pt, dash pattern={on 8pt off 2pt}, fill=white] (-2,-2) circle (1); 
\draw [line width=1pt, dash pattern={on 8pt off 2pt}, fill=white] (5,5) circle (1);  
\draw [line width=1pt, dash pattern={on 8pt off 2pt}, fill=white] (-5,5) circle (1); 
\draw [line width=1pt, dash pattern={on 8pt off 2pt}, fill=white] (5,-5) circle (1); 
\draw [line width=1pt, dash pattern={on 8pt off 2pt}, fill=white] (-5,-5) circle (1);

  \fill (2,1) circle (6pt); 
    \fill (2,-1) circle (6pt);
  \fill (-2,1) circle (6pt); 
    \fill (-2,-1) circle (6pt);
  \fill (1,2) circle (6pt); 
    \fill (1,-2) circle (6pt);
  \fill (-1,2) circle (6pt); 
    \fill (-1,-2) circle (6pt); 
    
      \fill (4,5) circle (6pt); 
    \fill (4,-5) circle (6pt);
  \fill (-4,5) circle (6pt); 
    \fill (-4,-5) circle (6pt);
  \fill (5,4) circle (6pt); 
    \fill (5,-4) circle (6pt);
  \fill (-5,4) circle (6pt); 
    \fill (-5,-4) circle (6pt); 
 
  \fill (45:3.8) circle (6pt); 
    \fill (45:6.1) circle (6pt);
  \fill (-45:3.8) circle (6pt); 
    \fill (-45:6.1) circle (6pt);
      \fill (90+45:3.8) circle (6pt); 
    \fill (90+45:6.1) circle (6pt);
      \fill (180+45:3.8) circle (6pt); 
    \fill (180+45:6.1) circle (6pt);    
  
\draw [line width=2pt, dashed] (2,2) circle (1);  
\draw [line width=2pt, dashed] (-2,2) circle (1); 
\draw [line width=2pt, dashed] (2,-2) circle (1); 
\draw [line width=2pt, dashed] (-2,-2) circle (1); 
\draw [line width=2pt, dashed] (5,5) circle (1);  
\draw [line width=2pt, dashed] (-5,5) circle (1); 
\draw [line width=2pt, dashed] (5,-5) circle (1); 
\draw [line width=2pt, dashed] (-5,-5) circle (1); 

 \end{tikzpicture}
  \hspace{0.4cm} 
 \begin{tikzpicture}
[scale=0.42]

\draw (0,0) circle (9);
\draw [line width=2pt] (-5,-5) -- (5,-5)  (-5,5) -- (5,5) ;
 \draw [line width=2pt] (-5,-5) -- (-5,5)  (5,-5) -- (5,5) ;
\draw [line width=2pt] (-2,-2) -- (2,-2)  (-2,2) -- (2,2) ;
 \draw [line width=2pt] (-2,-2) -- (-2,2)  (2,-2) -- (2,2) ; 
 
 
\draw [line width=2pt] (-2,-2)--(-5,-5) (2,2)--(5,5) (-2,2)--(-5,5)  (2,-2)--(5,-5);
  \fill (2,2) circle (6pt); 
  \fill (-2,2) circle (6pt); 
    \fill (2,-2) circle (6pt); 
      \fill (-2,-2) circle (6pt); 
   \fill (5,5) circle (6pt); 
      \fill (-5,5) circle (6pt); 
         \fill (5,-5) circle (6pt); 
            \fill (-5,-5) circle (6pt); 
 
\draw [line width=1pt, dash pattern={on 8pt off 2pt}] (0,0) circle (8.5) ;  
\fill[gray, opacity=0.1] (0,0) circle (8.5);
\draw [line width=1pt, dash pattern={on 8pt off 2pt}, fill=white] (0,0) circle (1.41); 
\draw [line width=1pt, dash pattern={on 8pt off 2pt}, fill=white] (1,3)--(3,4)--(-3,4)--(-1,3)--(1,3); 
\draw [line width=1pt, dash pattern={on 8pt off 2pt}, fill=white] (1,-3)--(3,-4)--(-3,-4)--(-1,-3)--(1,-3) ; 
\draw [line width=1pt, dash pattern={on 8pt off 2pt}, fill=white] (-3,1)--(-4,3)--(-4,-3)--(-3,-1)--(-3,1);  
\draw [line width=1pt, dash pattern={on 8pt off 2pt}, fill=white] (3,1)--(4,3)--(4,-3)--(3,-1)--(3,1); 

\draw [line width=2pt, dashed] (2,2)--(1,1) (1,3)--(3,1);  
\draw [line width=2pt, dashed] (-2,2)--(-1,1) (-3,1)--(-1,3); 
\draw [line width=2pt, dashed] (2,-2)--(1,-1) (1,-3)--(3,-1) ; 
\draw [line width=2pt, dashed] (-2,-2)--(-1,-1) (-1,-3)--(-3,-1) ; 
\draw [line width=2pt, dashed] (5,5)--(6,6) (4,3)--(5,5) (5,5)--(3,4) ;  
\draw [line width=2pt, dashed] (-5,5)--(-6,6) (-5,5)--(-4,3) (-5,5)--(-3,4) ; 
\draw [line width=2pt, dashed] (5,-5)--(6,-6) (5,-5)--(4,-3) (5,-5)--(3,-4) ; 
\draw [line width=2pt, dashed] (-5,-5)--(-6,-6) (-5,-5)--(-4,-3) (-5,-5)--(-3,-4) ; 

 \end{tikzpicture}
 
\caption{A fully regular map and some of its associated edge-biregular maps}
\label{fig11}
\end{figure}
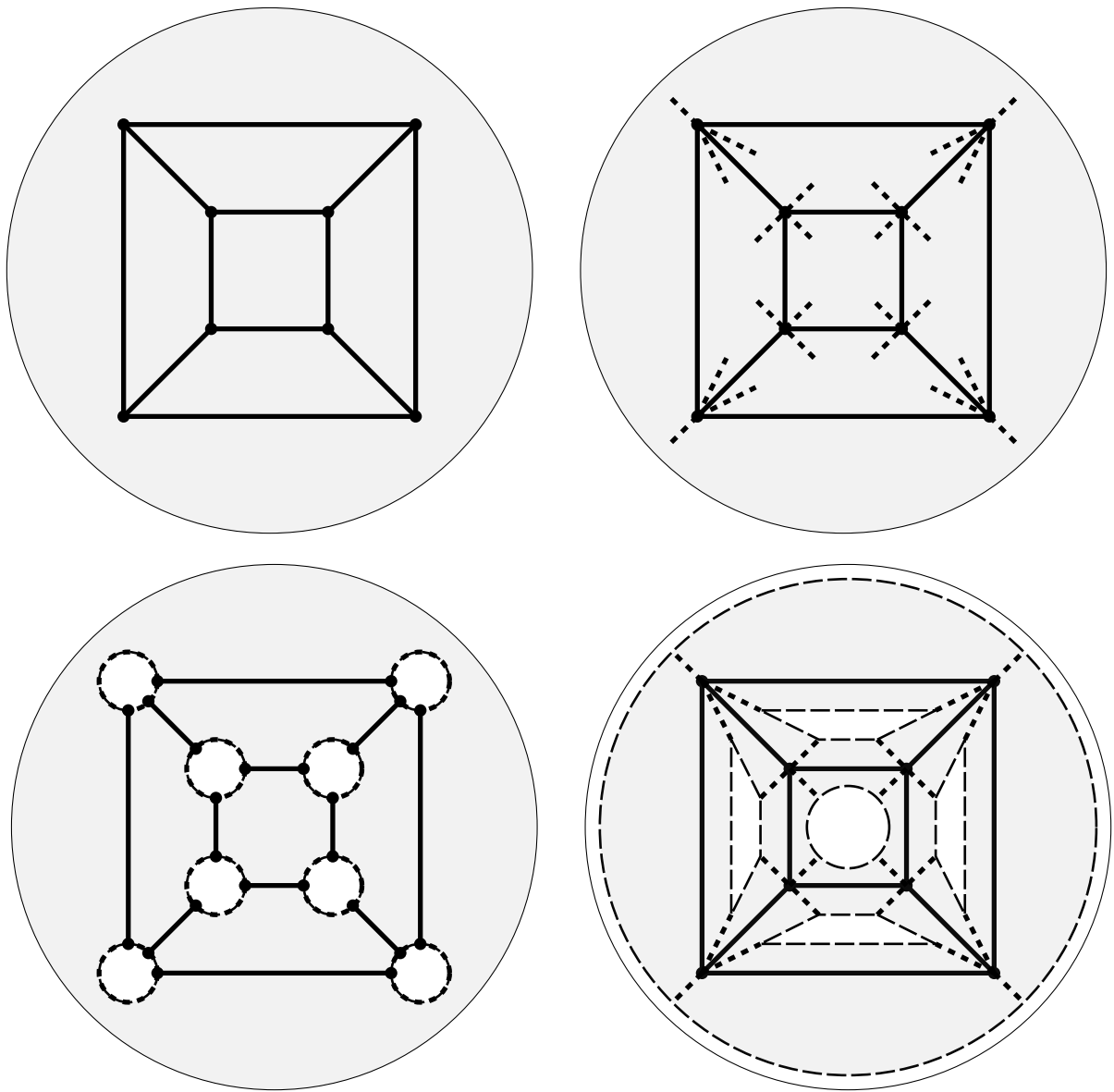

\smallskip

Having addressed maps with non-empty boundary components in the previous section, henceforth we will assume that the supporting surfaces for our edge-biregular maps are closed, that is without boundary.

\smallskip

There are other degeneracies arising from non-distinct generators, and we will now address these in turn. Up to duality and twinness we may assume that one of the non-distinct generators is $r_0$, and so as to avoid bold semi-edges, we assume the other is in $\{ \rho_0, \rho_2 \}$. 

In the case when $r_0=\rho_0$ we have $H = \langle r_0, r_2, \rho_2 \rangle $, the group for an edge-biregular map which has digonal faces. Maps of type $(k,2)$ are known to be fully regular, and are supported by either the sphere or the projective plane.

\begin{constr}\label{constr4}
The edge-biregular map $(H;\rho,r_2,\rho,\rho_2)$ where $\rho=r_0=\rho_0$ has an implied regular map $(G;r_0,r_2,r_1)$ where $r_1=\rho_2$ which is an embedding of the bold edges of the edge-biregular map, with the dashed edges deleted. This implied fully regular map also has digonal faces. In the other direction, a fully regular map with digonal faces can be made into an edge-biregular map by the addition of unshaded edges, each cutting each original digonal face into two alternate-edge-coloured digons.
\end{constr}

\smallskip

\begin{rem}
It may be tempting to think that any edge-biregular map $(H; \dots )$ with non-distinct generators, implies a natural regular map $(G; r_0, r_2, r_1)$ with $H \cong G$ and supported by the same surface. However, this is not the case as will now become clear.
\end{rem}

\smallskip

In the case where $r_0=\rho_2=\rho$ we have the edge-biregular map $(H;\rho,r_2,\rho_0,\rho)$ while the group $H = \langle r_0, r_2, \rho_0 \rangle$. Geometrically $r_0=\rho_2$ implies that (all) the edges in the bold orbit are loops. Also $r_0r_2=\rho_2 r_2$ has order dividing two and hence the vertices all have degree (2 or) $4$. Similarly $\rho_0\rho_2 = \rho_0 r_0$ has order dividing two, and so (assuming no further degeneracies occur) we have an edge-biregular map of type $(4,4)$. Now, the tessellation of type $(4,4)$ is Euclidean, indicating that the maps $(H;\rho,r_2,\rho_0,\rho)$ supported by a closed surface must be, depending on orientability, on either the Klein bottle or the torus. \textit{Note: We present a classification of proper edge-biregular maps on the torus and the Klein bottle in sections \ref{classtorus} and \ref{classKlein}.}

Treating $r_0$ and $r_2$ in the usual way, we consider the implications of $\rho_0$ being thought of as $r_1$, the reflection in a corner of a regular map. Informally this is like reducing (the marked one, and hence all) the dashed edges to negligable length, thereby identifying the endpoints of each dashed edge and stitching the remaining shaded flags together in the correspondingly natural way. The resulting fully regular map would therefore have faces which are digons and the underlying graph would consist of just one vertex. Thus the underlying graph has become a bouquet of loops, \textbf{but} this process has also drastically changed the underlying surface, as we might have expected if we remembered that the Klein bottle has no regular map. In the case of the Klein bottle, the resulting surface is the projective plane. However, in the case of a toroidal map, the object created by this process (which is equivalent to contracting a non-contractible cycle) is in fact a pseudo-surface, a sphere with one pair of antipodal points identified, in which case the resulting map is an example of a regular pinched map, see \cite{ABS} for further details. 

\bigskip

\section{Edge-biregular maps on closed surfaces where $\chi \ge 0$}

\subsection{Euler's formula}

Consider the edge-biregular map $(H; r_0, r_2, \rho_0, \rho_2)$ which has type $(k,l)$, and no semi-edges.

The stabiliser for the distinguished face is denoted $D_l : = \langle r_0, \rho_0 \rangle$ and is isomorphic to the dihedral group with $l$ elements, while the stabiliser for the distinguished vertex is called $D_k : = \langle r_2, \rho_2 \rangle$ and is isomorphic to the dihedral group with $k$ elements. The map thus has $\frac{|H|}{l}$ faces and $\frac{|H|}{k}$ vertices.

The stabiliser group of an edge is isomorphic to $V_4$ and, for the distinguished shaded edge is $\langle r_0,r_2 \rangle$, while for the distinguished unshaded edge the stabiliser is $\langle \rho_0, \rho_2 \rangle$. Hence the map has $\frac{2|H|}{4}$ edges.

Supposing that the map $(H; r_0, r_2, \rho_0, \rho_2)$ lies on a surface of Euler characteristic $\chi$, we apply the well-known Euler-Poincar\'{e} formula, which is useful when we come to classifying these maps on particular surfaces:
\begin{equation}
\chi = |H| (\frac{1}{k} - \frac{1}{2}	+\frac{1}{l})
\end{equation}

When applying Euler's formula to a map with semi-edges one should realise that the existence of semi-edges contribute nothing to the value of $\chi$.

\bigskip

\subsection{Edge-biregular maps of non-negative Euler characteristic}

\bigskip
In this section, we classify all edge-biregular maps on surfaces for which $\chi \in \{ 0,1,2 \}$. We note that Duarte's thesis \cite{D} includes a classification of 2-restrictedly-regular edge-bipartite hypermaps, that includes all {\textit{proper}} edge-biregular maps, for the sphere, the projective plane and the torus. Here we allow for the possibility of semi-edges, and we present a different approach for Euclidean edge-biregular maps. While Duarte's work allows for many different group presentations which describe the same Euclidean map, our approach standardises the presentation for such an edge-biregular map, and extends to include a classification for such maps on the Klein bottle.

\begin{thm}
If the proper edge-biregular map $M = (H; r_0,r_2,\rho_0,\rho_2)$ has non distinct generators, then it is supported by a surface which has non-negative Euler characteristic.
\end{thm}

\begin{proof}
Suppose that the generators $r_0,r_2,\rho_0,\rho_2$ are not distinct. The map is assumed to be proper, so there are no semi-edges which forces $r_0 \neq r_2 $ and $\rho_0 \neq \rho_2$. Up to duality, we may assume that one of the redundant generators is labelled $r_0$. This leaves two options, $r_0=\rho_0$, or $r_0=\rho_2$, which we will address in turn.

If $r_0=\rho_0$ then the element $r_0 \rho_0$ is the identity which means $l=2$ and the faces of the map are digons. Also note that $H = \langle r_0,r_2,\rho_0,\rho_2 \rangle = \langle r_2,\rho_0,\rho_2 \rangle$ and we have $[\rho_0,r_2]$ as well as $[\rho_0,\rho_2]$ so $\rho_0$ is central. Hence either $H = D_k$ or $H = D_k \times \langle \rho_0 \rangle$. Applying the Euler-Poincar\'{e} formula we see that these two cases correspond respectively to a single-vertex degree-$k$ map on the projective plane (remembering $4|k$), and a degree-$k$ dipole embedded in the sphere (for any even $k$).

If $r_0=\rho_2$, then the bold edges must be loops. Also we have $[\rho_0,\rho_2]=[\rho_0,r_0]$ so $l=4$ and $[r_0,r_2]=[\rho_2,r_2]$ so $k=4$ and hence the map type is $(4,4)$. This implies we are on a surface with Euler characteristic $\chi = 0$. If the surface is orientable then the map is on the torus, otherwise the map is supported by the Klein bottle.
\end{proof}

We can go further:
In the case where $r_0=\rho_2$, notice that $r_0$ is central, $H = \langle r_0,r_2,\rho_0 \rangle$, and also that $\langle r_2,\rho_0 \rangle$ is a dihedral group, with order, say, $2m$. Now, if $r_0 \in \langle r_2,\rho_0 \rangle$ then $H = \langle r_2, \rho_0 \rangle \cong D_{2m}$ where $m$ is even. Also, being central, $r_0=(r_2 \rho_0)^{m/2}$ which yields a relator of odd length, forcing the supporting surface to be non-orientable, the Klein bottle. If, on the other hand, $r_0 \notin \langle r_2, \rho_0 \rangle$ then $H = \langle r_2, \rho_0 \rangle \times \langle r_0 \rangle \cong D_{2m} \times C_2$ and, as the direct product of a dihedral group (with the presentation generated by two involutions) and a copy of $C_2$, there can be no relators of odd length. These maps are therefore supported by an orientable surface of Euler characteristic $0$, namely the torus. 

\smallskip

The workings above have shown that when $H$ has non-distinct generators $r_0,r_2,\rho_0,\rho_2$ then the closed supporting surface for the edge-biregular map has non-negative Euler characteristic $\chi \in \{ 0,1,2 \}$. This yields the following corollary on which we will rely later.

\begin{cor}
A proper edge-biregular map $M= (H; r_0,r_2,\rho_0,\rho_2)$, embedded on a surface where $\chi < 0$, has distinct generators $r_0,r_2,\rho_0,\rho_2$.
\end{cor}

\subsubsection{Edge-biregular maps on the torus}\label{classtorus}

Each regular map on the torus will have associated edge-biregular maps each created by inserting a semi-edge into each corner. Thus there will be edge-biregular maps of types $(12,6)$, $(6,12)$, $(8,8)$ derived by construction \ref{constr3} respectively from the well-known toroidal regular maps of types $(6,3)$, $(3,6)$ and $(4,4)$.

Turning our attention to proper edge-biregular maps supported by the torus, the requirement for even valency and face length makes it clear that they must all have type $(4,4)$. The infinite group $$ \Delta  \; =  \; \langle   \; R_0, R_2, P_0, P_2  \; | \; R_0^2, R_2^2, P_0^2, P_2^2, (R_0 R_2)^2, (P_0 P_2)^2, (R_0 P_0)^2, (R_2 P_2)^2  \;  \rangle$$ 
\noindent
describes the colour-preserving automorphism group of the alternate-edge-coloured infinite square grid, a tessellation of type $(4,4)$ on the Euclidean plane.

Two linearly independent translations of the plane determine a fundamental region for a torus. By marking an arbitrary point on the square grid, these two linearly independent translations will also generate a lattice of points on the plane. When the subgroup generated by the two translations is a normal subgroup $N_T$ of finite index in $\Delta$, the quotient of $\Delta$ by $N_T$ will then give $H$, the group of automorphisms for the edge-biregular map embedded on the torus.

A toroidal proper edge-biregular map, necessarily of type $(4,4)$, will thus have presentation:
$$H  \; = \; \langle  \; r_0, r_2, \rho_0, \rho_2  \; | \; r_0^2, r_2^2, \rho_0^2, \rho_2^2, (r_0 r_2)^2, (\rho_0 \rho_2)^2, (r_0 \rho_0)^2, (r_2 \rho_2)^2, \dots  \; \rangle $$
\noindent
where the dots indicate extra relators. The extra relators arise from $N_T$, the kernel of the epimorphism $\phi : \Delta \to H$ where $\phi : R_i \to r_i$ and $\phi : P_i \to \rho_i$ for each $i \in \{ 0, 2 \}$.

Since a map on the torus must be finite, each element of $H$ must have finite order. Specifically $r_0\rho_2$ must have finite order in $H$, let us say $a$. This means that $(R_0 P_2)^a$ is a translation in $N_T$. Since $a$ is, by definition, as small as possible, we may assume $(R_0 P_2)^a$ is one of two translations which generate the two-dimensional lattice of points, or equivalently $N_T$. Hence we may also assume that the other translation has the form $(R_0P_2)^b(R_2P_0)^c$ where $0 \leq b < a$ and $c>0$ is minimised. 

The group $N_T$ is normal in the group $\Delta$ if and only if all conjugates of each of the generators of $N_T$ are themselves members of the subgroup $N_T$. It is clear that this is the case for the generator $(R_0 P_2)^a$. In the case when $b=0$, all conjugates of $(R_0P_2)^b(R_2P_0)^c=(R_2P_0)^c$ are also within $N_T$, the lattice is rectangular, and $N_T = \langle (R_0 P_2)^a, (R_2P_0)^c \rangle$. 
However, when $b \neq 0$ then $N_T$ is normal if and only if we have both $((R_0P_2)^{ b}(R_2P_0)^c)^{R_0}= (R_0P_2)^{-b}(R_2P_0)^c \in N_T$ or equivalently $(R_0P_2)^{2b} \in N_T$, and also $(R_2P_0)^{2c} \in N_T$. This implies $a | 2b$ but $0 < b < a$ so we must have $2b=a$, meaning that the lattice, which is now also generated by $\langle (R_0P_2)^{ b}(R_2P_0)^c, (R_0P_2)^{ b}(R_2P_0)^{-c} \rangle$, is rhombic. 

This working can be summarised as follows, with examples of the corresponding maps shown in the diagrams in Figure \ref{fig12}:

\begin{prop}
Every proper toroidal edge-biregular map, up to duality and twinness, is determined by the group $H_{\dots}$ with one of the two following presentations:
$$H_{Rect} = \langle r_0, r_2, \rho_0, \rho_2  \; |  \; r_0^2, r_2^2, \rho_0^2, \rho_2^2, (r_0 r_2)^2, (\rho_0 \rho_2)^2, (r_0 \rho_0)^2, (r_2 \rho_2)^2, (r_0\rho_2)^a, (r_2 \rho_0)^c  \rangle $$
\noindent
or
$$H_{Rhomb} = \langle r_0, r_2, \rho_0, \rho_2  \; |  \; r_0^2, r_2^2, \rho_0^2, \rho_2^2, (r_0 r_2)^2, (\rho_0 \rho_2)^2, (r_0 \rho_0)^2, (r_2 \rho_2)^2, (r_0\rho_2)^{2b}, (r_0\rho_2)^b(r_2 \rho_0)^c  \rangle$$
\noindent
where $a$, $b$ and $c$ are positive integers.
\end{prop}

These maps are fully regular if and only if the lattice is in fact a square lattice, that is, in the case of $H_{Rect}$ when $a=c$, and in the case of $H_{Rhomb}$ when $b=c$. Informally, this can be observed by looking at Figure \ref{fig12} and noting the need for all closed straight-ahead walks to have the same length in a fully regular map. More formally, the map is fully regular if and only if $\psi : r_i \leftrightarrow \rho_i$ for $i \in \{ 0,2 \}$ is an automorphism, and inspection of the presentations of the groups quickly yields the same necessary and sufficient condition.

\begin{figure}
\centering
 \begin{tikzpicture}
 
\fill[gray, opacity=0.1] (0,0) rectangle(4,3);
 
  \foreach \x in {0,...,4}
    \draw [dashed, line width=2pt] (\x,0) -- (\x,3);
      \foreach \y in {0,...,3}
    \draw [line width=2pt] (0,\y) -- (4,\y);

\draw [->>, line width=2pt] (0,0)--(2,0);
\draw [->>, line width=2pt] (0,3)--(2,3);

\draw [->, dashed, line width=2pt] (0,0)--(0,1.5);
\draw [->, dashed, line width=2pt] (4,0)--(4,1.5);
  
 \end{tikzpicture}
 \hspace{2cm}
 \begin{tikzpicture}

 \begin{scope}
 \clip (0,0) -- (2,-3) -- (4,0) -- (2,3) -- (0,0);
\fill[gray, opacity=0.1] (0,0) -- (2,-3) -- (4,0) -- (2,3) -- (0,0);
 \foreach \x in {0,...,4}
    \draw [dashed, line width=2pt] (\x,-3) -- (\x,3);
      \foreach \y in {-2,...,2}
    \draw [line width=2pt] (0,\y) -- (4,\y);
\end{scope} 
 
 \draw [->>, dotted, line width=1pt] (0,0)--(1,-3/2);
  \draw [ dotted, line width=1pt] (1,-3/2)--(2,-3);
\draw [->>, dotted, line width=1pt] (2,3)--(3,3/2);
\draw [ dotted, line width=1pt] (4,0)--(3,3/2);

\draw [dotted, line width=1pt] (0,0)--(2,3);
\draw [dotted, line width=1pt] (2,-3)--(4,0);
\draw [->, dotted, line width=1pt] (0,0)--(1,3/2);
\draw [->, dotted, line width=1pt] (2,-3)--(3,-3/2);
           
 \end{tikzpicture}
\caption{Rectangular and rhombic toroidal edge-biregular maps for $a=4, b=2, c=3$}
\label{fig12}

 \end{figure}
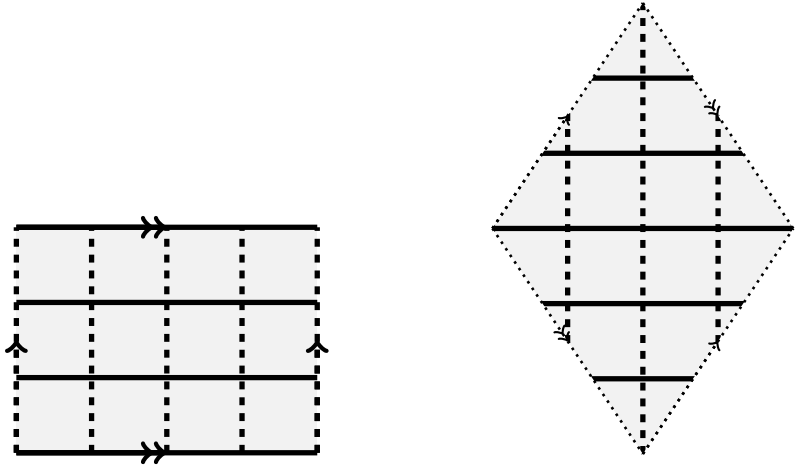

\subsubsection{Edge-biregular maps on the Klein bottle}\label{classKlein}

It is known that the Klein bottle supports no regular maps, so it follows that any edge-biregular map on this surface cannot have any semi-edges.

Since the Klein bottle is also Euclidean, an edge-biregular map on this surface is determined by forming the quotient of the infinite group $\Delta$ by a finite-index normal subgroup which we denote $N_K$.

The fundamental region for a Klein bottle is determined by a glide reflection and a translation in the direction perpendicular to the axis of reflection. See Coxeter and Moser, pp 40-43 in \cite{CM}, or for further details of all uniform maps on this surface, see Wilson \cite{W3}. For an edge-biregular map, the grid must be mapped to itself by the glide reflection, maintaining the same colouring of edges. Thus the axis of reflection must be orthogonal to the square grid, and up to duality and twinness, we may assume the glide reflection is 
$(R_2P_0)^a  R_0$
and the translation is
$(R_0P_2)^b$
where $a$ and $b$ are positive integers and as small as possible.

Suppose $N_K = \langle (R_2P_0)^a  R_0, (R_0P_2)^b \rangle$ is normal in $\Delta$, so that the quotient of $\Delta$ by $N_K$ will define an edge-biregular map on the Klein bottle. Conjugates of the translation are certainly in $N_K$ whereas $((R_2P_0)^a  R_0)^{P_2} = (R_2P_0)^a P_2R_0P_2 \in N_K$ if and only if $(R_0P_2)^2 \in N_K$, that is, if and only if $b|2$.

This is summarised in the following proposition, with examples shown in Figure \ref{fig13}.

\begin{prop}
Edge biregular maps on the Klein bottle are defined by the group $H$ with presentation:
$$H = \langle r_0, r_2, \rho_0, \rho_2  \; |  \; r_0^2, r_2^2, \rho_0^2, \rho_2^2, (r_0 r_2)^2, (\rho_0 \rho_2)^2, (r_0 \rho_0)^2, (r_2 \rho_2)^2, (r_2\rho_0)^a r_0, (r_0\rho_2)^b \rangle $$
\noindent
where $a$ is a positive integer and $b \in \{ 1,2 \}$.
\end{prop}

\begin{figure}
\centering
 \begin{tikzpicture}
 
\fill[gray, opacity=0.1] (0,0) rectangle(5,1);
 
  \foreach \x in {0,...,5}
    \draw [line width=2pt] (\x,0) -- (\x,1);
      \foreach \y in {0,1}
    \draw [dashed, line width=2pt] (0,\y) -- (5,\y);

\draw [->>, dashed,  line width=2pt] (0,0)--(2.5,0);
\draw [->>, dashed,  line width=2pt] (0,1)--(2.5,1);

\draw [->, line width=2pt] (0,0)--(0,0.5);
\draw [->, line width=2pt] (5,1)--(5,0.5);
  
 \end{tikzpicture}
 \hspace{2cm}
 \begin{tikzpicture} 

\fill[gray, opacity=0.1] (0,0) rectangle(5,2);
 
  \foreach \x in {0,...,5}
    \draw [line width=2pt] (\x,0) -- (\x,2);
      \foreach \y in {0,1,2}
    \draw [dashed, line width=2pt] (0,\y) -- (5,\y);

\draw [->>, dashed,  line width=2pt] (0,0)--(2.5,0);
\draw [->>, dashed,  line width=2pt] (0,2)--(2.5,2);

\draw [->, line width=2pt] (0,0)--(0,1);
\draw [->, line width=2pt] (5,2)--(5,1);

 \end{tikzpicture}
\caption{Edge-biregular maps on the Klein bottle for $a=5$ where $b=1$ and $b=2$.}
\label{fig13}

 \end{figure}
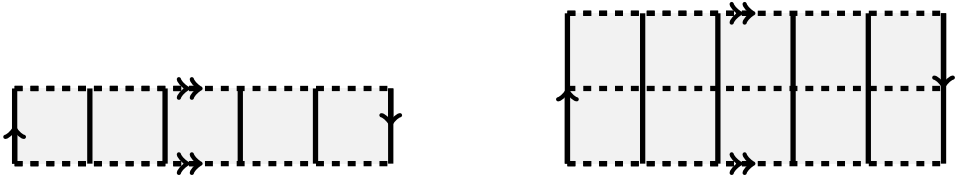

\subsubsection{Edge-biregular maps on the sphere and the projective plane}

\medskip

We will consider edge-biregular maps on these surfaces according to the number of orbits of semi-edges.

As we have seen earlier, any semistar of even valency can be embedded in the sphere to form an edge-biregular map.

An edge-biregular map with exactly one semi-edge orbit must come from a fully regular map by construction \ref{constr3}. It is well-known that the sphere supports fully regular maps of type $(3,3)$, $(3,4)$, $(4,3)$, $(3,5)$, $(5,3)$, $(2,m)$ and $(m,2)$ and so, by attaching a semi-edge at each corner, there are edge-biregular maps on the sphere of types $(6,6)$, $(6,8)$, $(8,6)$, $(6,10)$, $(10,6)$, $(4,2m)$ and $(2m,4)$. Similarly the projective plane supports fully regular maps of type $(3,4)$, $(4,3)$, $(3,5)$, $(5,3)$, $(2,m)$ and $(m,2)$ so on this surface there will be edge-biregular maps of types $(6,8)$, $(8,6)$, $(6,10)$, $(10,6)$, $(4,2m)$ and $(2m,4)$.

Finally, if an edge-biregular map contains no semi-edges, then (as now both the valency and the face length must be even) Euler's formula implies that on a sphere or a projective plane proper edge-biregular maps can only be embedded even cycles and their duals, of types $(2,2m)$ and $(2m,2)$ respectively.

\bigskip

\section{Classification of proper edge-biregular maps when $H$ is dihedral and $\chi < 0 $}

A full classification of edge-biregular maps for surfaces of Euler characteristic $\chi = -p$ where $p$ is prime will appear in a forthcoming paper by \v{S}ir\'{a}\v{n} and the author. Here we will conclude with a ``dihedral classification''. Henceforth we will be focussing on surfaces with negative Euler characteristic and so we may assume that $k \geq 4$ and $l \geq 4$.

\smallskip

Dihedral groups can be generated by two involutions, so it seems natural that some of our groups $H$ will themselves be dihedral. The following working gives a classification, up to twinness and duality, of edge-biregular maps on surfaces for which $\chi <0$ when the group $H = \langle r_0, r_2, \rho_0 , \rho_2  \rangle \cong D_{2m}$.

Since we know $V_4 \cong \langle r_0, r_2 \rangle \leq H$ (because there are no semi-edges) it is clear that $m$ is even. Therefore $H$ has a central involution, which we will call $z$. We also know that $V_4 \cong \langle \rho_0, \rho_2 \rangle \leq H$. Having dealt with special cases in the previous section, an assumption we can now make is that the generators of $H$, namely $r_0, r_2, \rho_0,$ and $\rho_2$, are distinct. For a dihedral group $D_{2m}$ to have at least four distinct involutions, we must have $m\geq 4$, which forces the group to be non-abelian. Also, note that every copy of $V_4$ as a proper subgroup in a dihedral group contains the unique central involution $z$ and, taking account of the generators being distinct, this means $z \in \{ r_0r_2, \rho_0\rho_2 \}$.

By our choice of colouring for the edge orbits, that is up to twinness, we may assume $z=r_0r_2$. We also have $z \in \{\rho_0, \rho_2, \rho_0\rho_2 \}$ which gives options which we address in turn.

\begin{enumerate}
\item
Suppose $z=\rho_0\rho_2$.

In this case $r_0r_2=\rho_0\rho_2$ and so $\rho_0 r_0=\rho_2 r_2$ which forces $k=l$. Also $\rho_2=\rho_0r_0r_2$ so $H = \langle \rho_0,r_0,r_2 \rangle = \langle \rho_0,r_0,r_0r_2 \rangle$. But we know that $\langle \rho_0,r_0 \rangle = D_l$ and so this leaves us with two possibilities:

Firstly, if $z=r_0r_2 \in D_l$ then $H=D_l = \langle \rho_0,r_0 \rangle \cong D_{2m}$ and the type is $(2m,2m)$. Notice, this map has just one face and one vertex, and is supported by a surface with Euler characteristic $\chi = 2-m$. By the fact the map has a single vertex, we also have $H= \langle r_2,\rho_2 \rangle$;

Second, if $z=r_0r_2 \notin D_l$ then, since $z$ is central, we get $H \cong D_l \times \langle z \rangle \cong D_{2m}$. By the uniqueness of the central involutory element, this can only happen if $D_l$ has trivial centre, that is if the order of $\rho_0r_0$ is odd, which happens when $\frac{l}{2}=\frac{m}{2}$ is odd. By considering the order of $(\rho_0r_0)z$ we see that $H= \langle \rho_0, r_0 \rangle  \langle z \rangle = \langle \rho_0, r_0z \rangle = \langle \rho_0, r_2 \rangle$. We also get, by remembering $\rho_2=\rho_0z$, that $H= \langle r_0, \rho_2 \rangle$. These presentations yield a map of type $(m,m)$ which thus has two faces and two vertices, supported by a surface where $\chi = 4-m$.

\item
Up to duality, the other case is when $z=\rho_0$. Now suppose $z=\rho_0$.

In this case we have $r_0r_2=\rho_0$ so $r_0\rho_0=r_2$ which is an involution. Thus $l=4$. Also we get $H = \langle \rho_2, r_2, r_0r_2 \rangle$. By similar reasoning to above this gives us two possibilities:

Firstly, if $z \in D_k$ we have $H = D_k = \langle \rho_2, r_2 \rangle \cong D_{2m}$, the map has a single vertex with quadrilteral faces, it is of type $(2m,4)$ and exists on a surface where $\chi=\frac{1}{2}(2-m)$;

Second, by the centrality of $z=r_0r_2 \notin D_k$ we have $H = D_k \times \langle z \rangle = \langle r_2,\rho_2 \rangle \times \langle z \rangle$. Also, by our assumption, $H \cong D_{2m}$ and so the order of $r_2\rho_2$ must be odd and $H = \langle zr_2,\rho_2 \rangle = \langle r_0, \rho_2 \rangle$. The map is of type $(m,4)$ where $\frac{m}{2}$ is odd, and thus the map has two vertices. This map occurs on a surface with $\chi=\frac{1}{2}(4-m)$.

Note that these (non-orientable) maps are not regular since regularity would require an automorphism of the group $H$ which swaps $\rho_0$ which is central in $H$ with $r_0$ which is not central.

\end{enumerate}

These results can be tabulated as follows giving a classification, up to duality and twinness, of edge-biregular maps on surfaces for which $\chi <0$ when the group $H$ is dihedral, $H = \langle r_0,r_2,\rho_0,\rho_2 \rangle \cong D_{2m}$ where $m$ is even and $m \geq 4$. For completeness, we note that dihedral edge-biregular maps on the sphere, projective plane, torus and Klein bottle also exist, and they are included in our earlier analysis.

\noindent
\begin{tabular}[\small]{ | c | c | c | c | c ||  c | c | c |}
\hline
& & & & & & & \\ [0.1ex]
Type & $H \cong D_{2m} $ & Relations & Conditions & $\chi <0$ & $V$ & $F$ & Regular\\[0.75ex]
\hline
& & & & & & & \\
$(2m,2m)$ & $\begin{array}{c}
\langle r_0,\rho_0 \rangle \\
\langle r_2,\rho_2 \rangle \end{array}$ & $r_0r_2=\rho_0\rho_2=(r_0\rho_0)^{\frac{m}{2}}=(r_2\rho_2)^{\frac{m}{2}}$ & None & $2-m$ & 1 & 1 & Yes\\
& & & & & & & \\
\hline
& & & & & & & \\
$(m,m)$ & $\begin{array}{c} \langle r_2,\rho_0 \rangle \\ \langle r_0,\rho_2 \rangle\end{array}$& $r_0r_2=\rho_0\rho_2=(r_2\rho_0)^{\frac{m}{2}}=(r_0\rho_2)^{\frac{m}{2}}$ & $\frac{m}{2}$ is odd & $4-m$ & 2 & 2 & Yes\\
& & & & & & & \\ 
\hline
& & & & & & & \\ 
$(2m,4)$ & $\langle r_2,\rho_2 \rangle$ & $r_0r_2=\rho_0=(r_2\rho_2)^{\frac{m}{2}}$ & None& $\frac{1}{2}(2-m)$ & 1 & $\frac{m}{2}$ & No\\
& & & & & & & \\ 
\hline
& & & & & & & \\ 
$(m,4)$ & $\langle r_0,\rho_2 \rangle$ & $r_0r_2=\rho_0=(r_0\rho_2)^{\frac{m}{2}}$ & $\frac{m}{2}$ is odd & $\frac{1}{2}(4-m)$ & 2 & $\frac{m}{2}$ & No \\
& & & & & & & \\ 
\hline
\end{tabular}

\begin{rem} It is worth noting that this classification, along with the earlier work, ensures that there is no closed surface which does not support edge-biregular maps. Specifically, all non-orientable surfaces with negative Euler characteristics support an edge-biregular map of type $(4(1-\chi),4)$, as well as its dual map of type $(4,4(1-\chi))$, and their twins, while the orientable surface with negative Euler characteristic $\chi$ supports an edge-biregular map of type $(2(2-\chi),2(2-\chi))$.
\end{rem}

\section{Concluding remarks}

This paper presents the background of edge-biregular maps, and addresses the existence of these maps on surfaces with boundary components and surfaces with non-negative Euler characteristic as well as, for other closed surfaces, a classification for edge-biregular maps with dihedral colour-preserving automorphism groups.

Edge-biregular maps on surfaces of negative prime Euler characteristic are classified in an upcoming paper by \v{S}ir\'{a}\v{n} and the author. Further research into edge-biregular maps is planned and the author also hopes to investigate the lesser-studied types of maps which arise from other index two subgroups of the full triangle groups.

\bigskip

\noindent{\bf Acknowledgement.} \ The author wishes to thank Jozef \v{S}ir\'{a}\v{n} for many helpful and enjoyable discussions during the preparation of this paper.

\end{document}